\documentclass[12pt,a4paper]{amsart}

\title[Computing Young's Natural Representations]{Computing Young's Natural Representations for Generalized Symmetric Groups}

\author{Koushik Paul}
\author{Götz Pfeiffer}
\address{School of Mathematical and Statistical Sciences, University of Galway,
  Ireland}
\email{k.paul2@universityofgalway.ie, goetz.pfeiffer@universityofgalway.ie}

\keywords{Symmetric Group; Hyperoctahedral Group; Generalized Symmetric Group; Matrix Representations}
\subjclass[2020]{Primary 20C30; 20C15}

\usepackage[euler-digits]{eulervm}
\usepackage[margin=1in]{geometry}
\usepackage{amssymb}\let\emptyset\varnothing
\usepackage{ytableau}\ytableausetup{smalltableaux,nobaseline}
\usepackage[hidelinks]{hyperref}

\newtheorem{theorem}{Theorem}[subsection]
\newtheorem{proposition}[theorem]{Proposition}
\newtheorem{lemma}[theorem]{Lemma}
\newtheorem{corollary}[theorem]{Corollary}
\theoremstyle{definition}
\newtheorem{definition}[theorem]{Definition}
\newtheorem{example}[theorem]{Example}
\theoremstyle{remark}
\newtheorem{remark}[theorem]{Remark}

\numberwithin{equation}{subsection}

\newcommand{\C}{\mathbb{C}}
\newcommand{\N}{\mathbb{N}}
\newcommand{\Z}{\mathbb{Z}}

\newcommand{\Symm}{\mathfrak{S}}
\newcommand{\Grin}{\mathfrak{S}^{(r)}}
\newcommand{\Hypo}{\mathfrak{H}}
\newcommand{\nnn}{[n]}

\newcommand{\diag}{\mathop{\mathrm{diag}}\nolimits}
\newcommand{\Ind}{\mathop{\mathrm{Ind}}\nolimits}
\newcommand{\Irr}{\mathop{\mathrm{Irr}}\nolimits}
\newcommand{\GL}{\mathop{\mathrm{GL}}\nolimits}

\newcommand{\Size}[1]{\lvert #1 \rvert}
\newcommand{\Span}[1]{\langle #1 \rangle}
\newcommand{\pair}[2]{\genfrac[]{0pt}{1}{#1}{#2}}
\newcommand{\pairx}[2]{\genfrac[]{0pt}{}{#1}{#2}}
\newcommand{\sgn}[1]{\varepsilon_{#1}}
\newcommand{\SYT}[1]{\mathrm{SYT}_{#1}}

\begin{document}

\begin{abstract}
  We provide an algorithmic framework for the computation of explicit
  representing matrices for all irreducible representations of a
  generalized symmetric group $\Grin_n$, i.e., a wreath product of
  cyclic group of order $r$ with the symmetric group $\Symm_n$.  The
  basic building block for this framework is the Specht matrix, a
  matrix with entries $0$ and $\pm1$, defined in terms of pairs of
  certain words.  Combinatorial objects like Young diagrams and Young
  tableaus arise naturally from this setup.  In the case $r = 1$, we
  recover Young's natural representations of the symmetric group.  For
  general $r$, a suitable notion of pairs of $r$-words is used to
  extend the construction to generalized symmetric groups.
  Separately, for $r = 2$, where $\Grin_n$ is the Weyl group of type
  $B_n$, a different construction is based on a notion of pairs of
  biwords.
\end{abstract}

\dedicatory{To the memory of Richard Parker (1953--2024)}

\maketitle

\section{Introduction}

By Wedderburn's Theorem, the group algebra $\mathbb{C}G$ of a finite group
$G$ over the complex numbers $\mathbb{C}$ is isomorphic to a direct sum of
full matrix rings over $\C$.  Hence $\mathbb{C}G$ has a $\mathbb{C}$-basis
consisting of matrix units.  Identifying such a basis in $\mathbb{C}G$ is
equivalent to fixing explicit representing matrices for all the irreducible
representations of $G$.  In the case where $G$ is a cyclic group, this process
of changing the basis of $\mathbb{C}G$ is called Fast Fourier Transform.

If $G$ is the symmetric group $\Symm_n$, it is well-known that certain
elements constructed from sums, or signed sums, over the elements of
so-called Young subgroups of $\Symm_n$ are very close to being suitable
matrix units for a basis of $\mathbb{C} \Symm_n$.  Derived from this observation are the various descriptions of representing matrices for $\Symm_n$, known
as Young's natural form,
Young's seminormal form,
and Young's orthogonal form, where perhaps the natural form, despite its name, is the lesser known of them all.

There are plenty of excellent and comprehensive introductions to the
representation theory of the symmetric group, like the books by
Rutherford~\cite{Rutherford48}, James and Kerber~\cite{JaKe81},
Fulton~\cite{Fulton97}, and Sagan~\cite{Sagan2001}, where also historical
accounts on the pioneering work of Frobenius, Young, and Specht can be found.
Young's semi-normal form of the matrices for the irreducible modules of
the symmetric group has been generalized to other classes of groups,
and their Hecke algebras, for instance by Hoefsmit~\cite{Hoefsmit74} to the Hecke algebra of the Weyl group of type $B_n$, and by Ariki and Koike~\cite{AriKoi94} to the generic Hecke algebra of the generalized symmetric group.

Here, we use the
notion of a Specht matrix, as introduced in~\cite{WWZ17}, based on the
action of $\Symm_n$ on words of length $n$, and then its
action on certain pairs of such words, to reconstruct Young's natural
form. Combinatorial objects like partitions, or standard Young tableaus,
arise naturally in the process, which results in a uniform formula for all
representing matrices in all irreducible representations of $\Symm_n$, see
Theorem~\ref{thm:symm-rep}.

Using suitably generalized notions of words, and pairs of words, we can generalize this
approach to the monomial groups $\Grin_n$, also known as generalized
symmetric groups, or as complex reflection groups $G(r, 1, n)$ in the
Shepard--Todd classification \cite{ShepardTodd54}. This results in a uniform
formula for all representing matrices in all irreducible representations of
$\Grin_n$, see Theorem~\ref{thm:mono-rep}.

Moreover, using another variation of the notion of words, and pairs of words,
we can generalize this approach to the hyperoctahedral groups $\Hypo_n$, also
known as Weyl groups of type $B_n$, and as monomial groups $\Grin_n$ for
$r = 2$. This results in another uniform formula for all representing matrices
in all irreducible representations of $\Hypo_n$, see
Theorem~\ref{thm:hypo-rep}.

This paper is based on the first author's Ph.D.\ thesis~\cite{KPaulThesis}.
The algorithms have been implemented in GAP~\cite{GAP4} and are available on
\texttt{github}~\cite{specht}.
Our treatment is elementary.  The resulting matrices have integral entries,
suitable for modular reduction.  Possible extensions of our methods to other
groups and a version of our formulas for Hecke algebras are the subject of
ongoing research.

\textbf{Notation:} $\N = \{1,2,3, \dotsc\}$ and $\nnn = \{1, \dots, n\}$, for
$n \in \N \cup \{0\}$.  A function $f \colon X \to Y$, in particular a
permutation $f: X \to X$, acts from the right on its argument $x \in X$: we
often write $x.f$ for the image of $x$ under $f$.  The \emph{composition}
$fg$ of functions $f$ and $g$ is defined as $x.{(fg)} = (x.f).g$.
We write $\sgn{\sigma}$ for the \emph{sign} of a permutation $\sigma$.
All modules and vector spaces are over the field $\C$ of complex numbers.


\section{Symmetric Groups}\label{sec:symm}

We describe a construction of explicit matrices for the irreducible modules
of the symmetric group $\Symm_n$.

\subsection{Partitions and Diagrams.}
\label{sec:A-part}

A \emph{integer partition} $\lambda$ (or a \emph{partition}, for short) is a
multiset of nonnegative integers $\lambda_i$, called the \emph{parts} of $\lambda$,
usually written in decreasing order, omitting zeros, as
\begin{align*}
  \lambda = (\lambda_1, \dots, \lambda_l) \quad \text{ with }
  \lambda_1 \geq \dots \geq \lambda_l > 0\text.
\end{align*}
We say that $\ell(\lambda):= l$ is the \emph{length} of $\lambda$, and for
convenience set $\lambda_j = 0$ for $j > l$.  If
$\Size{\lambda}:= \lambda_1 + \dots + \lambda_l = n$, we say that $\lambda$ is a partition
\emph{of} $n$ and write $\lambda \vdash n$.  We write $\emptyset$ for the empty partition, the unique partition of~$0$.  We denote $\Lambda_n := \{\lambda : \lambda \vdash n\}$ and assume that $\Lambda_n$ is ordered lexicographically.
\begin{example}
  $\Lambda_5 = \{( 1, 1, 1, 1, 1 ), ( 2, 1, 1, 1 ), ( 2, 2, 1 ), ( 3, 1, 1 ), ( 3, 2 ), ( 4, 1 ), ( 5 )\}$.
\end{example}

We call a finite subset of
$\N \times \N$ a \emph{diagram}.  One assigns to the partition $\lambda$ the
\emph{Young diagram}
\begin{align*}
  D(\lambda) := \{(i, j) \in \N^2 : j \leq \lambda_i\}\text.
\end{align*}
Note that $D(\lambda) \subseteq {\nnn}^2$ if $\lambda \vdash n$.  Clearly,
$\lambda$ can be recovered from its diagram.  The \emph{transpose} of a
diagram $D$ is the diagram $D^t = \{(j, i) : (i, j) \in D\}$.  The diagram
$D(\lambda)^t$ is the diagram of a partition $\lambda^t$, called the
\emph{transpose} of~$\lambda$.  A Young diagram is commonly represented by an
array of square boxes in the positions $(i, j) \in D$.
\begin{example}
  The partition $\lambda = (2,1,1)$ has diagram
  $\smash{\raisebox{6pt}{\tiny\ydiagram{2,1,1}}}$ with transpose
  $\smash{\raisebox{6pt}{\tiny\ydiagram{3,1}}}$ corresponding to
  the partition $\lambda^t = (3,1)$.
\end{example}

\subsection{Permutations and Cycle Type.}
\label{sec:A-perm}

A \emph{permutation} of a finite set $X$ is a bijective map
$\sigma \colon X \to X$.  The group of all permutations of $X$ is called the
\emph{symmetric group} of $X$, denoted by $\Symm_X$.  Here we are mostly
interested in the case $X = \nnn$, for a nonnegative integer $n$, and write
$\Symm_n$ for $\Symm_{\nnn}$.

The cycles of a
permutation $\sigma \in \Symm_n$ form a partition of the set $\nnn$, whence
their lengths form an integer partition of $n$, called the \emph{cycle type} of $\sigma$.
Two permutations in $\Symm_n$ are conjugate if and only if they have the same
cycle type.  Hence, the set $\Lambda_n$
parameterizes both the conjugacy classes and the irreducible representations of
$\Symm_n$.

In the following, we provide an explicit construction of
irreducible representations of $\Symm_n$, one for each $\lambda \in
\Lambda_n$. For this, the permutation action of $\Symm_n$ on the cosets of its
subgroup
\begin{align*}
  \Symm_{\lambda} := \Symm_{\lambda_1} \times  \dots \times \Symm_{\lambda_l}\text,
\end{align*}
for $\lambda = (\lambda_1, \dots, \lambda_l)$, plays a crucial role.

\subsection{Words and Partitions.}
\label{sec:A-words}

Let $n \in \N$, and let $A$ be a finite alphabet.  The symmetric group
$\Symm_n$ acts on the set $A^n$ of \emph{words} of length $n$ by permuting
the letters of a word:
\begin{align*}
  (a_1 \dotsm a_n).\sigma = a_{1.\sigma^{-1}} \dotsm a_{n.\sigma^{-1}}\text,
\end{align*}
for $a_1, \dots, a_n \in A$ and $\sigma \in \Symm_n$.
We may regard a word $w = w_1 \dotsm w_n$ as a map
\begin{align*}
  w \colon \nnn \to A\text,\quad i \mapsto i.w = w_i
\end{align*}
with \emph{inverse image map}
\begin{align*}
w^{*} \colon A \to 2^{\nnn}\text,\quad
  a \mapsto a.w^{*} = \{i \in [n] : w_i = a\}\text.
\end{align*}
Then $\Symm_n$ acts on
$A^n$ by \emph{left inverse composition} (i.e., $w.\sigma = \sigma^{-1} w$),
and it acts on the set $\{w^{*} : w \in A^n\}$ by
\emph{right composition}\footnote{Here we are abusing notation $\sigma$
for the permutation $\sigma \colon 2^{\nnn} \to 2^{\nnn}$ arising from the action of
$\Symm_n$ on the
power set $2^{\nnn}$, via
$  J.\sigma = \{j.\sigma: j \in J\}$
for $J \subseteq \nnn$ and $\sigma \in \Symm_n$.}
(i.e., $w^{*}. \sigma = w^{*} \sigma$),
so that $(w.\sigma)^{*}= w^{*}.\sigma$ for all $w \in A^n$, $\sigma \in \Symm_n$.

The \emph{stabilizer} in $\Symm_n$ of the word $w \in A^n$ is the subgroup
$\Symm(w^{*}):= \prod_{a \in A} \Symm(a.{w^{*}})$, the direct product of the
symmetric groups on the (non-empty) fibers of~$w$.  The
$\Symm_n$-\emph{orbit} of $w$ consists of all its \emph{rearrangements}.  The
multiset of the letter frequencies $\Size{a.w^{*}}$ of $w$ corresponds to a partition of
$n$, which we call the \emph{shape} of~$w$. If $w$ has shape $\lambda$, then
its stabilizer $\Symm(w^*)$ is conjugate to $\Symm_{\lambda}$ and thus the action of $\Symm_n$ on the
set of rearrangements of $w$ is equivalent to its action on the cosets of
$\Symm_{\lambda}$.

Now let $A = \nnn$ and fix $\lambda \vdash n$.
We define the \emph{canonical word} $w_{\lambda} \in A^n$ of shape $\lambda$ as
\begin{align*}
  w_{\lambda} := 1^{\lambda_1} 2^{\lambda_2} \dotsm n^{\lambda_n}\text,
\end{align*}
that is $\lambda_1$ letters $1$ followed by $\lambda_2$ letters $2$
and so on.  We denote
by $X_{\lambda} := w_{\lambda}.\Symm_n$ the set of rearrangements of $w_{\lambda}$.
Then, as $\Symm_n$-set, $X_{\lambda}$ is
isomorphic to the right cosets of
$\Symm_{\lambda}$ in~$\Symm_n$.  For $\sigma \in \Symm_n$, we denote by $[\sigma]_{X_{\lambda}}^{X_{\lambda}}$ the matrix of the action of $\sigma$ on the
permutation module $V^{\lambda}$ with basis $X_{\lambda}$.  Then
\begin{align}\label{eq:symm-perm-mat}
  [v.\sigma]_{X_{\lambda}} = [v]_{X_{\lambda}}\, [\sigma]_{X_{\lambda}}^{X_{\lambda}},
\end{align}
for all $v \in V^{\lambda}$, $\sigma \in \Symm_n$.

\subsection{Pairs of Words and Tableaus.}
\label{sec:A-pairs}

Let $A = [n]$.  We consider the action of $\Symm_n$ on the set
$A^{2 \times n} = A^n \times A^n$ of pairs of words over the alphabet~$A$.
We write $\pair{x}{y}$ for the pair of words $x = x_1 \dotsm x_n$,
$y = y_1 \dotsm y_n \in A^n$ and regard it as a map
\begin{align*}
  \pair{x}{y} \colon [n] \to A^2\text,\quad
  i \mapsto \pair{x_i}{y_i}\text,
\end{align*}
where $A^2 = \{\pair{a}{b} : a, b \in A\}$.
Its \emph{inverse image function} is
\begin{align*}
  \pair{x}{y}^{*}\colon A^2 \to 2^{\nnn}\text, \qquad
  \pair{a}{b} \mapsto
  \{i \in \nnn : \pair{x_i}{y_i} = \pair{a}{b}\}
  = a.{x^{*}} \cap b.{y^{*}}\text.
\end{align*}
Note that $a.x^* = \coprod_{b \in A} \pair{a}{b}.\pair{x}{y}^*$
and $b.y^* = \coprod_{a \in A} \pair{a}{b}.\pair{x}{y}^*$.
So both $x^{*}$ and $y^{*}$
can be recovered from $\pair{x}{y}^{*}$, and thus the map $\pair{x}{y} \mapsto \pair{x}{y}^*$
is a bijection.

The action of $\Symm_n$ on $A^n$ induces an $\Symm_n$-action on the
set of pairs $\{\pair{x}{y} : x, y \in A^n\}$ via
$\pair{x}{y}.\sigma = \pair{x.\sigma}{y.\sigma}$.  Moreover, $\Symm_n$
acts by \emph{right composition} on the inverse images,
\begin{align*}
\pair{x}{y}^{*}.\sigma = \pair{x}{y}^{*} \sigma = (\pair{a}{b}
\mapsto (a.{x^{*}} \cap b.{y^{*}}).\sigma)\text,
\end{align*}
where
$\pair{x}{y}^{*}.\sigma = (\pair{x}{y}.\sigma)^{*}$, showing that the
bijection $\pair{x}{y} \mapsto \pair{x}{y}^*$ is $\Symm_n$-equivariant.

The stabilizer in $\Symm_n$ of the pair $\pair{x}{y} \in A^{2 \times n}$ is
$\Symm(\pair{x}{y}^{*}) = \Symm(x^{*}) \cap \Symm(y^{*})$.
\begin{example}
  The stabilizer in $\Symm_4$ of the pair $\pair{1123}{1112}$ has size $2$,
  due to the repetition of column $\pair{1}{1}$, whereas $\Symm(\pair{1321}{1112}^*) = 1$.
\end{example}
The pairs $\pair{x}{y}$ with trivial stabilizer form
the \emph{free
  component}
\begin{align*}
(A^{2 \times n})^{\#} := \{
\pair{x}{y} \in A^{2 \times n} :
\Symm(\pair{x}{y}^{*}) = 1\}
\end{align*}
of $A^{2 \times n}$,
where
$\Symm(\pair{x}{y}^{*}) = 1$ if and only if $\Size{a.{x^{*}} \cap b.{y^{*}}} \leq 1$ for
all $\pair{a}{b} \in A^2$.
For any pair of subsets $X, Y \subseteq A^n$ that are closed under the action of $\Symm_n$, we denote by $(X \times Y)^{\#} = \{\pair{x}{y} \in X \times Y :
\Symm(\pair{x}{y}^{*}) = 1\}$
the \emph{free component} of $X \times Y$.

Clearly, $\Symm_n$ acts freely on $(X \times Y)^{\#}$.  If this set forms a single $\Symm_n$-orbit, we can and will identify its elements with the permutations in $\Symm_n$.
For this, we now classify those pairs of
$\Symm_n$-orbits $X, Y \subseteq A^n$ for which $\Symm_n$ acts
transitively on the free component
$(X \times Y)^{\#}$.  As any set $X$ of rearrangements of a word $x$ of shape  $\lambda \vdash n$ is isomorphic to
$X_{\lambda}$,  it
suffices to consider $\Symm_n$-sets of the form $X_{\lambda} \times X_{\eta}$, for $\lambda, \eta \vdash n$.
We
furthermore define the \emph{diagram} of a pair $\pair{x}{y}$ as
$\{(a, b) : \pair{a}{b}.\pair{x}{y}^* \neq \emptyset \} \subseteq \nnn^2$,
i.e., as the set of pairs of letters that occur as a column $\pair{a}{b}$ in
the pair of words $\pair{x}{y}$, regarded as a subset of~$[n]^2$.
\begin{example}
  The pair $\pair{1321}{1112}$ has diagram $\smash{\raisebox{6pt}{\tiny\ydiagram{2,1,1}}}$, whereas $\pair{1123}{1112}$ has diagram $\smash{\raisebox{6pt}{\tiny\ydiagram{1,1,1+1}}}$.
\end{example}
\begin{proposition}\label{pro:symm-free-trans}
  Let $\lambda$ and $\eta$ be partitions of $n$.  Then $\Symm_n$ acts
  transitively on the free component of $X_{\lambda} \times X_{\eta}$ if and
  only if $\eta = \lambda^t$.  In that case, the diagram of each pair
  $\pair{x}{y} \in (X_{\lambda} \times X_{\lambda^t})^{\#}$ is the Young diagram
  $D(\lambda)$ of shape~$\lambda$.
\end{proposition}

\begin{proof}
Let $\lambda' = \eta^t$ and let $x \in X_{\lambda}$ and $y \in X_{\eta}$.
  Then $\lambda_1$ is number of letters $1$ in $x$, and $\lambda_1'$ is the
  number of distinct letters in~$y$.  Consider the lexicographical order on
  the set of all partitions of $n$.

  If $\lambda' < \lambda$ then, without loss of generality,
  $\lambda_1' < \lambda_1$: There are more letters $a = 1$ in $x$ than
  there are distinct letters in $y$, so by the Pigeonhole Principle,
  $\Size{a.{x^{*}} \cap b.{y^{*}}} > 1$ for one letter $b$ in~$y$. It
  follows that $(X_{\lambda} \times X_{\eta})^{\#} = \emptyset$ in
  this case.

  If $\lambda' > \lambda$ then, without loss of generality,
  $\lambda_1' > \lambda_1$: There are less letters $a = 1$ in $x$ than there
  are distinct letters in $y$.  It follows\footnote{There is a letter $b$ in
    $y$ so that $\pair{a}{b}$ does not occur in $\pair{x}{y}$.  But
    $\pair{c}{b}$ does occur in $\pair{x}{y}$ for at least one letter $c$ in
    $x$.  But since there are not more letters $c$ than letters $a$ in $x$
    there must be a letter $d$ in $y$ such that $\pair{c}{d}$ does not occur,
    but $\pair{a}{d}$ does.}  that there are letters $c$ in $x$ and $b, d$ in
  $y$ such that none of the pairs $\pair{a}{b}$ and $\pair{c}{d}$ does occur
  as a column in $\pair{x}{y}$ whereas both $\pair{a}{d}$ and $\pair{c}{b}$
  do.  Let $\pair{x'}{y'}$ be the pair obtained from $\pair{x}{y}$ by
  replacing the columns $\pair{a}{d}$ and $\pair{c}{b}$ by the columns
  $\pair{a}{b}$ and $\pair{c}{d}$:
  \begin{align*}
    \pair{x}{y} =
    \pair{\dotsm a \dotsm c \dotsm}{\dotsm d \dotsm b \dotsm}
    \leadsto
    \pair{x'}{y'} =
    \pair{\dotsm a \dotsm c \dotsm}{\dotsm b \dotsm d \dotsm}\text.
  \end{align*}
  Then the pairs $\pair{x}{y}$ and $\pair{x'}{y'}$ are
  not in the same $\Symm_n$-orbit.  But if $\pair{x}{y}$ lies in the
  free component then so does $\pair{x'}{y'}$.  Hence $\Symm_n$ does not
  act transitively on the free component in this case.

  However, if $\lambda' = \lambda$ then all pairs $\pair{x}{y}$ in
  the free component have the same diagram: a Young diagram of shape
  $\lambda$.
\end{proof}

Note that, if $\pair{x}{y} \in
(X_{\lambda} \times X_{\lambda^t})^{\#}$
then
$\Size{\pair{a}{b}. \pair{x}{y}^*}
= 1$ for all $(a, b) \in D(\lambda)$.
So we can replace each such set $\pair{a}{b}. \pair{x}{y}^*$ by its single element and regard $\pair{x}{y}^*$ as a \emph{(Young) tableau} $T$ of shape~$\lambda$:
\begin{align*}
  T \colon D(\lambda) \to \nnn\text,\, (a, b) \mapsto i\text{, if }
  \pair{a}{b}. \pair{x}{y}^* = \{i\}\text.
\end{align*}
Then, for $a \in \nnn$, the $a$th \emph{row} of
$T$ is $(T(a, 1), \dots, T(a, \lambda_a))$,
the positions $i \in \nnn$ such that $x_i = a$ and
$y_i = b$, for $b = 1, \dots, \lambda_a$, in that order.  Denote by $T^{\flat}$ the
concatenation of the rows of $T$.  Then $T^{\flat}$ is a permutation
of $\nnn$.  We will write $\pair{x}{y}^{\flat}$ for $T^{\flat}$ if $T$ is the tableau obtained from $\pair{x}{y}^*$, when $\pair{x}{y} \in (X_{\lambda} \times X_{\lambda^t})^{\#}$.  Clearly,
the map $\pair{x}{y} \mapsto \pair{x}{y}^{\flat}$ is
an equivariant bijection between
$(X_{\lambda} \times X_{\lambda^t})^{\#}$ and $\Symm_n$.

Following~\cite{WWZ17}, we now define the {Specht matrix} $M_{\lambda}$ for
$\lambda \vdash n$ as follows.
\begin{definition}\label{def:symm-specht}
  Let $\lambda \vdash n$.   The \emph{Specht matrix} for $\lambda$ is the matrix $M_{\lambda} = (m^{\lambda}_{yx})$,
with rows labelled by $y \in X_{\lambda^t}$ and columns labelled by $x \in X_{\lambda}$,
  where
  \begin{align*}
    m^{\lambda}_{yx} =
    \begin{cases}
      \sgn{\pairx{x}{y}^{\flat}},& \text{if }
                       \pair{x}{y} \in (X_{\lambda} \times X_{\lambda^t})^{\#}\text,\\
      0, & \text{else.}
    \end{cases}
  \end{align*}
\end{definition}

\begin{remark}\label{rem:symm-myx}
  Note that
    $m_{yx}^{\lambda} = \sgn{\sigma}\, m_{y.\sigma,x.\sigma}^{\lambda}$
  for all $\sigma \in \Symm_n$, $x \in X_{\lambda}$, and $y \in X_{\lambda^t}$.
\end{remark}

\begin{example} Let
  $\lambda = (2,1,1)$, $\lambda^t = (3,1)$.
  Note how, e.g., the pair $\pair{1321}{1112}$
  corresponds to the tableau $T = \pair{1321}{1112}^*
  = \smash{\raisebox{6pt}{\tiny\ytableaushort{14,3,2}}}$\,
  and permutation $T^{\flat} = [1,4,3,2] = (2,4)$ with $\sgn{T^{\flat}} = -1$.
  \begin{align*}
    \begin{array}{c|cccccccccccc}
      &\textcolor{cyan}{1123}&1132&\textcolor{cyan}{1213}&\textcolor{cyan}{1231}&1312&1321&2113&2131&2311&3112&3121&3211 \\\hline
  \textcolor{cyan}{1112}&   . &  . &  . & \textcolor{cyan}{+1} &  . & -1 &  . & -1 & +1 &  . & +1 & -1 \\
  \textcolor{cyan}{1121}&   . &  . & \textcolor{cyan}{-1} &  . & +1 &  . & +1 &  . & -1 & -1 &  . & +1 \\
  \textcolor{cyan}{1211}&  \textcolor{cyan}{+1} & -1 &  . &  . &  . &  . & -1 & +1 &  . & +1 & -1 &  . \\
  2111&  -1 & +1 & +1 & -1 & -1 & +1 &  . &  . &  . &  . &  . &  . \\
    \end{array}
  \end{align*}
\end{example}

\subsection{Irreducible Modules.}
\label{sec:symm-irr-mod}

Following Fulton~\cite[Section~7.2]{Fulton97}, we construct an irreducible
$\Symm_n$-module $S^{\lambda}$ for each partition $\lambda \vdash n$, and
then identify it with the row space of the Specht matrix~$M_{\lambda}$.

Let $\lambda$ be a partition of $n$, and
recall that $V^{\lambda}$ is the
permutation module with basis $X_{\lambda}$, the rearrangements of $w_{\lambda}$.
For each word $y \in X_{\lambda^t}$, set
\begin{align*}
  c_y := \sum_{\sigma \in \Symm(y^*)} \sgn{\sigma} \sigma \in \C \Symm_n\text.
\end{align*}
Note that
$c_y^2 = \Size{\Symm_{\lambda}}\, c_y$.
Then,
for each tableau $T = \pair{x}{y}^{*}$ of shape $\lambda$ (that is, for each pair $\pair{x}{y} \in (X_{\lambda} \times X_{\lambda^t})^{\#}$),
we define a vector $v_T \in V^{\lambda}$ as $v_T := \sgn{T^{\flat}}\, x.c_y$.  Then
\begin{align*}
  v_T
  = \sum_{\sigma \in \Symm(y^{*})} \sgn{(T.\sigma)^{\flat}}\, x.\sigma
  = \sum_{x'} \sgn{\pair{\smash{x'}}{y}^{\flat}}\, x'
  = \sum_{x' \in X_{\lambda}} m^{\lambda}_{yx'}\, x'\text,
\end{align*}
where the second sum is over all rearrangements $x'$ of $x$ such that $\pair{x'}{y} \in (X_{\lambda} \times X_{\lambda^t})^{\#}$.
It follows that $v_T$ only depends on $y \in X_{\lambda^t}$,
and that the coefficient vector $[v_T]_{X_{\lambda}} = (m^{\lambda}_{yx})_{x \in X_{\lambda}}$ is the $y$-row of the Specht matrix $M_{\lambda}$.
We set $v_y := v_T$ for any tableau $T = \pair{x}{y}^*$.

\begin{proposition}\label{pro:symm-S-lambda}
  Let $S^{\lambda} := \Span{v_y : y \in X_{\lambda^t}}_{\C}$.  Then
  $S^{\lambda}$ is a $\C \Symm_n$-module.  In fact,   $S^{\lambda} = v_y \C \Symm_n$, for any  $y \in X_{\lambda^t}$.
\end{proposition}

\begin{proof} For each $\sigma \in \Symm_n$ and each $y \in X_{\lambda^t}$, we have
  \begin{align*}
    v_y.\sigma
    = \sum_{x\in X_{\lambda}} m_{yx}^{\lambda}\, x.\sigma
    = \sgn{\sigma} \sum_{x\in X_{\lambda}} m_{y.\sigma,x.\sigma}^{\lambda}\, x.\sigma
    = \sgn{\sigma} v_{y.\sigma}\text,
  \end{align*}
by Remark~\ref{rem:symm-myx}.
\end{proof}

\begin{corollary}\label{cor:symm-ny-action}
  Let $\lambda, \lambda' \vdash n$ be such that $\lambda < \lambda'$. Then, for all $y \in X_{\lambda^t}$,
  \begin{enumerate}
  \item[(i)] $V^{\lambda}.c_y = S^{\lambda}.c_y = \C v_y \neq 0$, and
  \item [(ii)] $V^{\lambda'}.c_y = S^{\lambda'}.c_y = 0$.
  \end{enumerate}
\end{corollary}

\begin{proof}
  (i) Let $y \in X_{\lambda^t}$.  Then $S^{\lambda} \subseteq V^{\lambda}$ implies $\Span{v_y}
  \subseteq S^{\lambda}.c_y \subseteq V^{\lambda}.c_y = \Span{v_y}$.

  (ii)   Let $y \in X_{\lambda^t}$.  If the pair $\pair{x}{y} \in X_{\lambda'} \times X_{\lambda^t}$ does not lie in a free orbit then it contains a repeated column $\pair{a}{b} \in A^2$, and hence  the stabilizer $\Symm(y^{*})$ contains a transposition $\tau$ that fixes $x$.
  From $\sgn{\tau} = -1$ follows $\tau c_y = -c_y$,
  whence $x.c_y = (x.\tau).c_y = x.(\tau c_y) = -x.c_y$, and so $x.c_y = 0$.
  In the proof of Proposition~\ref{pro:symm-free-trans}, it has been observed that the free part of $X_{\lambda'} \times X_{\lambda^t}$ is empty if $\lambda < \lambda'$.  Hence $x.c_y = 0$ for all $x \in X_{\lambda'}$.
\end{proof}

The following criterion for indecomposability is straightforward.

\begin{lemma}\label{la:indec}
  If, for any finite group $G$ and a $G$-module $V$, there is an element
  $c \in \C G$ such that $V.c = \Span{u}$ for some $u \in V$, and if
  $u.\C G = V$ then $V$ is indecomposable.
\end{lemma}

\begin{theorem}\label{thm:symm-complete}
  The modules $S^{\lambda}$ for $\lambda \vdash n$ form a complete set of
  pairwise non-isomorphic irreducible $\C \Symm_n$-modules.
\end{theorem}

\begin{proof}
  As $v_y. \C \Symm_n = S^{\lambda}$ by Proposition~\ref{pro:symm-S-lambda}
  and
  $S^{\lambda}.c_y = \Span{v_y}$ by Corollary~\ref{cor:symm-ny-action}(i), the
  $\Symm_n$-module $S^{\lambda}$ is irreducible by Lemma~\ref{la:indec}.
  Moreover, by
  Corollary~\ref{cor:symm-ny-action}(ii), $S^{\lambda'}$ is not isomorphic to
  $S^{\lambda}$ if $\lambda < \lambda'$.
\end{proof}

\subsection{Representing Matrices and Standard Tableaus.}
Let $\lambda \vdash n$.
An element $\sigma \in \Symm_n$ acts on the $\Symm_n$-module $V^{\lambda}$
with basis $X_{\lambda}$ as a matrix $[\sigma]_{X_{\lambda}}^{X_{\lambda}}$
according to \eqref{eq:symm-perm-mat}.
We now identify a basis $B_{\lambda}$ of the submodule $S^{\lambda}$
and derive a formula for the representing matrices $[\sigma]_{B_{\lambda}}^{B_{\lambda}}$ in terms of the Specht matrix $M^{\lambda}$ and $[\sigma]_{X_{\lambda}}^{X_{\lambda}}$ as
Theorem~\ref{thm:symm-rep} below.

Our basis $B_{\lambda}$ will be defined in terms of standard Young
tableaus.
A tableau $T$ is called a \emph{standard (Young) tableau}, if
its entries are increasing along its rows (left to right)
and its columns (top to bottom).
\begin{definition}\label{def:std-pair}
  A pair of words $\pair{x}{y} \in A^{2 \times n}$ is a \emph{standard pair}
  if (i)  the map $x \colon [n] \to A$ is strictly increasing on the preimages
  $b.y^{*}$, for all $b \in A$, and (ii)
  the map $y \colon \nnn \to A$ is strictly increasing on the preimages $a.{x^{*}}$,
for all $a \in A$.
\end{definition}

Note that,  if $\pair{x}{y} \in (X_{\lambda} \times X_{\lambda^t})^{\#}$ is a standard pair then $T = \pair{x}{y}^{*}$ is a standard tableau.
We denote by $\SYT{\lambda}$ the set of all standard tableaus $T$ of shape
$\lambda$.

We assume that both $A^2$ and the set $A^{2 \times n}$
of words over $A^2$ are ordered lexicographically.  This order induces an order on the tableaus of shape $\lambda$, and in particular on $\SYT{\lambda}$.

\begin{example}
  For $\lambda = (3,2)$, we have $\SYT{\lambda}$ ordered as follows:
  \begin{align*}
    \pair{11122}{12312} <
    \pair{11212}{12132} <
    \pair{11221}{12123} <
    \pair{12112}{11232} <
    \pair{12121}{11223}\text{, i.e., }
    \smash{\raisebox{4pt}{\tiny\ytableaushort{123,45}}} <
    \smash{\raisebox{4pt}{\tiny\ytableaushort{124,35}}} <
    \smash{\raisebox{4pt}{\tiny\ytableaushort{125,34}}} <
    \smash{\raisebox{4pt}{\tiny\ytableaushort{134,25}}} <
    \smash{\raisebox{4pt}{\tiny\ytableaushort{135,24}}}\text.
  \end{align*}
\end{example}

Note that, with respect to this order, the standard pairs
$\pair{x}{y} \in (X_{\lambda} \times X_{\lambda^t})^{\#}$ are characterized
by the property
\begin{align}\label{eq:smallest-tab}
  \pair{x}{y} < \pair{x}{y}.\sigma\text{, for all }
  \sigma \in \Symm(x^*)\text{ and for all }\sigma \in \Symm(y^*)\text.
\end{align}
It follows that the projection maps from $\SYT{\lambda}$ onto $X_{\lambda}$ and $X_{\lambda^t}$, given by
$\pair{x}{y}^* \mapsto x$ and
$\pair{x}{y}^* \mapsto y$, are both injective. We denote the images of those projections by $X_{\lambda}^{\heartsuit}$ and $X_{\lambda^t}^{\heartsuit}$, respectively.  We now set
\begin{align*}
  B_{\lambda} := \{v_y : y \in X_{\lambda^t}^{\heartsuit}\}
  = \{v_{T} : T \in \SYT{\lambda}\}.
\end{align*}

\begin{lemma}\label{la:syt-indep}
  The set $B_{\lambda} \subseteq S^{\lambda}$ is linearly independent.
\end{lemma}

\begin{proof}
  Assuming that $X_{\lambda}^{\heartsuit}$ and $X_{\lambda^t}^{\heartsuit}$ each inherit the order from $\SYT{\lambda}$, due to property~\eqref{eq:smallest-tab},
  the submatrix $M_{\lambda}^{\heartsuit} = (m^{\lambda}_{yx})$ of $M_{\lambda}$,
  with rows labelled by $y \in X_{\lambda^t}^{\heartsuit}$ and columns labelled by $x \in X_{\lambda}^{\heartsuit}$, is up to signs upper unitriangular, hence invertible.
\end{proof}

\begin{example}
  For $\lambda = (3,2)$, the following table gives $M_{\lambda}^{\heartsuit}$, where the non-diagonal entry $-1$ corresponds to the non-standard tableau $T = \pair{12121}{12312}^* = \smash{\raisebox{4pt}{\tiny\ytableaushort{153,42}}}$ with $\sgn{T^{\flat}} = -1$.
  \begin{align*}
    \begin{array}{c|ccccc}
      &11122&11212&11221&12112&12121 \\\hline
12312& +1 & . & . & . & -1 \\
12132& . & -1 & . & . & . \\
12123& . & . & +1 & . & . \\
11232& . & . & . & +1 & . \\
11223& . & . & . & . & -1 \\
    \end{array}
  \end{align*}
  Note that here, $(M_{\lambda}^{\heartsuit})^{-1} = M_{\lambda}^{\heartsuit}$.
\end{example}

It turns out that $B_{\lambda}$ is in fact a basis of $S^{\lambda}$; see
Proposition~\ref{pro:symm-b-basis}.
We use this fact here for the proof of the main theorem in this section.

\begin{theorem}\label{thm:symm-rep}
  Let $\lambda \vdash n$.  Then
  \begin{align*}
    [\sigma]^{B_{\lambda}}_{B_{\lambda}} = (M_{\lambda}\, [\sigma]^{X_{\lambda}}_{X_{\lambda}})^{\heartsuit}\, (M_{\lambda}^{\heartsuit})^{-1}\text,
  \end{align*}
  for all $\sigma \in \Symm_n$.
\end{theorem}

\begin{proof}
Note that $[v]_{B_{\lambda}}\, M_{\lambda}^{\heartsuit} = [v]_{X_{\lambda}}^{\heartsuit}$
      for all $v \in S^{\lambda}$, as
    $[v_y]_{B_{\lambda}}\, M_{\lambda}^{\heartsuit} = [v_y]_{X_{\lambda}}^{\heartsuit}$ is the $y$-row of $M_{\lambda}^{\heartsuit}$, for all basis vectors $v_y \in B_{\lambda}$.
  Therefore,
    \begin{align*}
      [v_y . \sigma]_{B_{\lambda}}\, M_{\lambda}^{\heartsuit}
      = [v_y.\sigma]_{X_{\lambda}}^{\heartsuit}
      = ([v_y]_{X_{\lambda}}\,[\sigma]^{X_{\lambda}}_{X_{\lambda}})^{\heartsuit}\text,
    \end{align*}
    the $y$-row of the matrix $[\sigma]^{B_{\lambda}}_{B_{\lambda}}\, M_{\lambda}^{\heartsuit} = (M_{\lambda}\, [\sigma]^{X_{\lambda}}_{X_{\lambda}})^{\heartsuit}$.
\end{proof}

\begin{example}
  For $\lambda = (3,2)$, Theorem~\ref{thm:symm-rep} yields
  \begin{align*}
[(1,2)]_{B_{\lambda}}^{B_{\lambda}} &=
    {\tiny\arraycolsep3pt
    \left(
    \begin{array}{rrrrr}
   1 & 0 & 0 &-1 &-1 \\
   0 & \phantom{-}1 & 0 & 1 & 0 \\
   0 & 0 & \phantom{-}1 & 0 & 1 \\
   0 & 0 & 0 &-1 & 0 \\
   0 & 0 & 0 & 0 &-1 \\
    \end{array}
    \right)}\text,&
[(1,2,3,4,5)]_{B_{\lambda}}^{B_{\lambda}} &=
    {\tiny\arraycolsep3pt
    \left(
    \begin{array}{rrrrr}
   0 & \phantom{-}1 & 0 & 0 & 0 \\
   0 & 0 & 1 & 0 & 0 \\
   1 & 0 &-1 &-1 &-1 \\
   0 & 0 &-1 & 0 &-1 \\
   0 & 1 & 1 & 1 & 1 \\
    \end{array}
    \right)}\text.
  \end{align*}
\end{example}

\subsection{A Global Argument.}
We show that $B_{\lambda}$ spans $S^{\lambda}$ by considering all
$\lambda \vdash n$ at once.
\begin{proposition}[{\cite[Section~7.2]{Fulton97}}]\label{pro:symm-b-basis}
  For each $\lambda \vdash n$,  the set $B_{\lambda}$ is a basis of $S^{\lambda}$.
\end{proposition}
\begin{proof}
  Let $f_{\lambda} = \Size{\SYT{\lambda}} = \Size{B_{\lambda}}$ and let
  $d_{\lambda} = \dim S_{\lambda}$.  Then $f_{\lambda} \leq d_{\lambda}$ by
  Lemma~\ref{la:syt-indep}.  We have
  $\sum_{\lambda \vdash n} d_{\lambda}^2 = n!$ by Wedderburn's
  Theorem~\cite[(1.17)]{Isaacs94} and
  $\sum_{\lambda \vdash n} f_{\lambda}^2 = n!$ by the
  RSK-correspondence~\cite[Section~4.3]{Fulton97}.  Thus
  $\sum_{\lambda} d_{\lambda}^2 - f_{\lambda}^2 = 0$, and this is only
  possible if $d_{\lambda} - f_{\lambda} = 0$, i.e., if
  $\dim S^{\lambda} = \Size{B_{\lambda}}$, for all $\lambda \vdash n$.
\end{proof}


\section{Monomial Groups}\label{sec:mono}

The monomial group $\Grin_n$ is the group of all monomial
$n \times n$-matrices whose nonzero entries are $r$th roots of unity.
We describe a construction of explicit matrices for all irreducible modules of
$\Grin_n$.  For this, let $n, r \in \N$ be fixed throughout this
section, and let $\omega = e^{2\pi i/r}$, a primitive $r$th root of unity.
We denote by
\begin{align*}
  \mu_r = \{\omega^k : k \in \Z\} = \{\omega^0, \dots, \omega^{r-1}\}
\end{align*}
the set of all $r$th roots of unity.

\subsection{Multipartitions and Multidiagrams.}
\label{sec:mono-partition}

An \emph{$r$-partition} $\lambda$ is a sequence of $r$ partitions,
$\lambda = (\lambda^{(0)}, \dots, \lambda^{(r-1)})$. If
$\sum_{i=0}^{r-1} \Size{\lambda^{(i)}} = n$, we say that $\lambda$ is a
multipartition of $n$ and write $\lambda \vdash^r n$.
We denote by $\Lambda_n^{(r)}$ the set of all $r$-partitions of~$n$.

An $r$-diagram $D$ is a sequence of $r$ diagrams,
$D = (D^{(0)}, \dots, D^{(r-1)})$.  The \emph{Young diagram} of the $r$-partition
$\lambda$ is the $r$-diagram
$D(\lambda) = (D(\lambda^{(0)}), \dots, D(\lambda^{(r-1)}))$.  The
\emph{transpose} of the $r$-diagram $D$ is the $r$-diagram
$D^t = ((D^{(0)})^t, \dots, (D^{(r-1)})^t)$.  The transpose of the
multidiagram $D(\lambda)$ is the Young diagram of an $r$-partition of $n$, the
\emph{transpose} $\lambda^t = ((\lambda^{(0)})^t, \dots, (\lambda^{(r-1)})^t)$ of the $r$-partition~$\lambda$.

\begin{example}
  The $3$-partition $\lambda = ((2,1), (2), (1,1))$ has diagram
  $(\smash{\raisebox{4pt}{\tiny\ydiagram{2,1}}}, \smash{\tiny\ydiagram{2}}, \smash{\raisebox{4pt}{\tiny\ydiagram{1,1}}})$,
  with transpose
  $(\smash{\raisebox{4pt}{\tiny\ydiagram{2,1}}}, \smash{\raisebox{4pt}{\tiny\ydiagram{1,1}}}, \smash{\tiny\ydiagram{2}})$
  corresponding to the $3$-partition $\lambda^t = ((2,1), (1,1), (2))$.
\end{example}

For a monomial $n \times n$-matrix $\sigma \in \Grin_n$, let
$\alpha_i = \omega^{k_i}$ be the unique nonzero element in the $i$th row of
$\sigma$.  Then $\sigma$ is the product of the diagonal matrix
$\hat{\sigma}:= \diag(\alpha_1, \dots, \alpha_n)$ and the matrix of a
uniquely determined permutation~$\overline{\sigma} \in \Symm_n$.  In this
way, $\Grin_n$ is isomorphic to a semidirect product
$\Delta^{(r)}_n \rtimes \Symm_n$, where $\Delta^{(r)}_n \cong \mu_r^n$ is the
subgroup of all diagonal matrices in $\Grin_n$, on which $\Symm_n$ acts by
permuting the factors.  For $j \in \nnn$, we denote
\begin{align*}
  t_j := \diag(1, \dots, 1, \omega, 1, \dots, 1)\text,
\end{align*}
the diagonal matrix with
entry $\omega$ in position $j$.  Then
$\Delta^{(r)}_n = \Span{t_1, \dots, t_n}$ is an abelian group of order $r^n$.

To each cycle $(i_1, i_2, \dots, i_l)$ of $\overline{\sigma}$
corresponds the \emph{cycle product} $\alpha_{i_1} \alpha_{i_2} \dotsm \alpha_{i_l} \in \mu_r$.  The \emph{cycle type} of $\sigma$ is the
multipartition $\lambda = (\lambda^{(0)}, \dots, \lambda^{(r-1)})$, where
$\lambda^{(k)}$ is the multiset of cycle lengths of $\overline{\sigma}$ whose cycle product is $\omega^k$, for $k = 0, \dots, r-1$.

Then two matrices in $\Grin_n$ are conjugate in $\Grin_n$ if and only if they
have the same cycle type (see, e.g., \cite[Theorem~4.2.8]{JaKe81}).  Hence,
the set $\Lambda_n^{(r)}$ parameterizes both the conjugacy classes and the
irreducible representations of~$\Grin_n$.

\subsection{Irreducible Characters.}
\label{sec:mono-irr}

We construct the set $\Irr(\Grin_n)$ of irreducible characters of
$\Grin_n = \Delta^{(r)}_n \rtimes \Symm_n$ following \cite[Section~5.5.4]{GePf00}
using Clifford theory.  Recall that $\Delta^{(r)}_n$  is
a direct product of $n$ cyclic groups $\Span{t_j}$ of order $r$.

We get a particular kind of irreducible character of $\Grin_n$
as follows.  First, for $0 \leq k < r$, we define a linear character
\begin{align}\label{eq:eta-nk}
  \eta_{n}^{(k)} \colon \Delta^{(r)}_n \to \C^{*},\qquad
  t_j \mapsto \omega^k,\, j \in \nnn\text.
\end{align}
This character can be extended to a linear character
$\overline{\eta}_{n}^{(k)}$ of $\Grin_n$ by setting
$\overline{\eta}_{n}^{(k)}(\sigma) = 1$ for $\sigma \in \Symm_n$.  Second, let
$\lambda$ be a partition of $n$, and $\chi^{\lambda}$ the character of the
corresponding irreducible $\Symm_n$-module $S^{\lambda}$.  Composition with
the canonical surjection $\Grin_n \to \Symm_n$ turns this into an irreducible
$\Grin_n$-character $\tilde{\chi}^{\lambda}$.  Furthermore, for each $k$, the
tensor product $\overline{\eta}_{n}^{(k)} \otimes \tilde{\chi}^{\lambda}$ is
an irreducible character of $\Grin_n$.

In general,
for any $r$-tuple
\begin{align*}
\underline{n} = (n^{(0)}, \dots, n^{(r-1)})
\end{align*}
with $n^{(0)}, \dots, n^{(r-1)} \geq 0$
and $n^{(0)} + \dots + n^{(r-1)} = n$, we let $I^{(k)} \subseteq \{1, \dots, n\}$,
$0 \leq k < r$, be pairwise disjoint subsets such that $\coprod_{i=0}^{k} I^{(i)} = \{1, \dots, \sum_{i=0}^k n^{(i)}\}$.
Then $\Delta^{(r)}_n = \Delta^{(r)}_{n^{(0)}} \times \dots \times \Delta^{(r)}_{n^{(r-1)}}$,
where the $k$th factor is $\Delta^{(r)}_{n^{(k)}} = \Span{t_j : j \in I^{(k)}}$.  We
define a linear character
$\eta_{\underline{n}} \colon \Delta^{(r)}_n \to \C^{*}$ as
\begin{align}\label{eq:mono-char-eta}
\eta_{\underline{n}} = \eta_{n^{(0)}}^{(0)} \boxtimes \dots \boxtimes \eta_{n^{(r-1)}}^{(r-1)}\text,
\end{align}
where $\eta_{n^{(k)}}^{(k)} \colon \Delta^{(r)}_{n^{(k)}} \to \C^*$ is the character from \eqref{eq:eta-nk} and $\boxtimes$ denotes the outer
tensor product, i.e., $\eta_{\underline{n}}(t_j) = \omega^k$ if $j \in I^{(k)}$.
Since $\Grin_n$ acts on $\Delta^{(r)}_n$ by permuting the factors $\Span{t_j}$,
the characters $\{\eta_{\underline{n}} : \sum \underline{n} = n\}$ form a complete set of representatives
of the $\Grin_n$-orbits on $\Irr(\Delta^{(r)}_n)$.  It thus follows from Clifford's Theorem~\cite[(6.5)]{Isaacs94} that
\begin{align*}
  \Irr(\Grin_n) = \coprod_{\sum\underline{n} = n} \Irr(\Grin_n)_{\underline{n}}\text,
\end{align*}
where
$\Irr(\Grin_n)_{\underline{n}}$ denotes the set of all irreducible characters in $\Irr(\Grin_n)$
whose restriction to $\Delta^{(r)}_n$ contains $\eta_{\underline{n}}$.

Now fix an $r$-tuple $\underline{n} = (n^{(0)}, \dots, n^{(r-1)})$.
The stabilizer of $\eta_{\underline{n}}$ in $\Grin_n$ is the subgroup
$\Grin_{\underline{n}} \leq \Grin_n$ which via the diagonal embedding
$\GL_{n^{(0)}}(\C) \times \dots \times \GL_{n^{(r-1)}}(\C) \leq
\GL_n(\C)$ is isomorphic to the direct product
$\Grin_{n^{(0)}} \times \dots \times \Grin_{n^{(r-1)}}$.  The factor
$\Grin_{n^{(k)}}$ contains the normal subgroup
$\Delta^{(r)}_{n^{(k)}} = \Span{t_i \mid i \in I^{(k)}}$ and the subgroup
$\Grin_{n^{(k)}} \cap \Symm_n$ is isomorphic to
$\Symm_{n^{(k)}}$.  In fact,
$\Grin_{n^{(k)}} = \Delta^{(r)}_{n^{(k)}} \rtimes \Symm_{n^{(k)}}$.
Let $\lambda = (\lambda^{(0)}, \dots, \lambda^{(r-1)})$
be an $r$-partition of $n$ such that $\lambda^{(k)} \vdash n^{(k)}$ for all $k$.
Then $\overline{\eta}_{n^{(k)}}^{(k)} \otimes \tilde{\chi}^{\lambda^{(k)}}$
is an irreducible character of $\Grin_{n^{(k)}}$.
Denote by $\Irr(\Grin_{\underline{n}})_{\underline{n}}$ the set of irreducible
characters of $\Grin_{\underline{n}}$ whose restriction to
$\Delta^{(r)}_n$ contains $\eta_{\underline{n}}$.
Now $\eta_{\underline{n}}$ extends to the character
\begin{align*}
\overline{\eta}_{\underline{n}}
= \overline{\eta}_{n^{(0)}}^{(0)} \boxtimes \dots \boxtimes \overline{\eta}_{n^{(r-1)}}^{(r-1)}
\end{align*}
of $\Grin_{\underline{n}}$.
Thus, \cite[(6.17)]{Isaacs94} yields that
\begin{align*}
  \Irr(\Grin_{\underline{n}})_{\underline{n}} = \{ (\overline{\eta}_{n^{(0)}}^{(0)} \otimes \tilde{\chi}^{\lambda^{(0)}}) \boxtimes \dots \boxtimes (\overline{\eta}_{n^{(r-1)}}^{(r-1)} \otimes \tilde{\chi}^{\lambda^{(r-1)}}) : \lambda^{(k)} \vdash n^{(k)},\, 0\leq k < r\}\text.
\end{align*}
Finally, by~\cite[(6.11)]{Isaacs94}, induction from $\Grin_{\underline{n}}$ to $\Grin_n$ gives
a bijection from $\Irr(\Grin_{\underline{n}})_{\underline{n}}$ to
$\Irr(\Grin_{n})_{\underline{n}}$, mapping $ (\overline{\eta}_{n^{(0)}}^{(0)} \otimes \tilde{\chi}^{\lambda^{(0)}}) \boxtimes \dots \boxtimes (\overline{\eta}_{n^{(r-1)}}^{(r-1)} \otimes \tilde{\chi}^{\lambda^{(r-1)}})$ to an irreducible
character
\begin{align}\label{eq:mono-chi-lambda}
\chi^{\lambda} :=
  \left( (\overline{\eta}_{n^{(0)}}^{(0)} \otimes \tilde{\chi}^{\lambda^{(0)}}) \boxtimes \dots \boxtimes (\overline{\eta}_{n^{(r-1)}}^{(r-1)} \otimes \tilde{\chi}^{\lambda^{(r-1)}})\right)^{\Grin_n}
\end{align}
of $\Grin_n$.  Overall, we obtain that
\begin{align}\label{eq:mono-irr}
  \Irr(\Grin_n) = \{\chi^{\lambda} : \lambda \vdash^{r} n \}\text.
\end{align}

\subsection{Multiwords and Multipartitions.}
\label{sec:multiwords}

As before, let $A = \nnn$.
An \emph{$r$-multiword} (or \emph{$r$-word} for short) is a word $w$ over the \emph{$r$-alphabet} $A_r:= A \times \mu_r$,
regarded as a map
\begin{align*}
  w \colon [n] \to A_r, \qquad i \mapsto i.w = w_i,
\end{align*}
where each letter $w_i$ is a pair $(a_i, \omega^{k_i})$ that can be identified with the number $a_i \omega^{k_i} \in \C$.
We call $a_i$ the \emph{radius}, and $k_i$ the \emph{phase} of the letter $w_i$.

The \emph{inverse image map} of an $r$-word $w$ is the map
\begin{align*}
  w^{*} \colon A_r \to 2^{\nnn}, \qquad
  (a, \omega^k) \mapsto \{i \in \nnn : w_i = (a, \omega^k)\},
\end{align*}
which we identify with the
$r$-tuple $(w_{(0)}^{*}, \dots, w_{(r-1)}^{*})$ of inverse image maps
\begin{align*}
  w_{(k)}^{*} \colon A \to 2^{\nnn}, \qquad
  a \mapsto \{i \in \nnn : w_i = (a, \omega^k)\}\text.
\end{align*}

The symmetric group $\Symm_n$ acts on the set $A_r^n$ of $r$-words of length
$n$ via \emph{left inverse composition} (i.e., $w.\sigma = \sigma^{-1} w$),
and it acts on the set $\{w^* : w \in A_r^n\}$ by \emph{right composition}:
(i.e., $w^{*}.\sigma = w^{*}\sigma$ and
$w_{(k)}^{*}.\sigma = w_{(k)}^{*}\sigma$), so that
$(w. \sigma)^* = w^*.\sigma$ for all $w \in A_r^n$, $\sigma \in \Symm_n$.
As before, the $\Symm_n$-orbit of an $r$-word $w \in A_r^n$ consists of all its rearrangements.
The stabilizer in $\Symm_n$ of $w$ is the subgroup $\Symm(w^*) = \prod_{\omega^k \in \mu_r} \prod_{a \in A} \Symm((a, \omega^k).w^*)$, the direct product of the symmetric groups on the fibers of $w$.
For each $\omega^k \in \mu_r$, the multiset
of the letter frequencies $\Size{(a, \omega^k).w^*}$ of $w$ corresponds to a partition $\lambda^{(k)}$, such that $\lambda = (\lambda^{(0)}, \dots, \lambda^{(r-1)})$ is an $r$-partition of $n$, which we call the \emph{shape} of $w$.  If $w$ has shape $\lambda$ then its stabilizer in $\Symm_n$ is conjugate to the subgroup
\begin{align*}
  \Symm_{\lambda} := \Symm_{\lambda^{(0)}} \times \dots \times \Symm_{\lambda^{(r-1)}}
\end{align*}
of $\Symm_n$.  We set $\Grin_{\lambda} := \Delta^{(r)}_n \Symm_{\lambda}$.

\begin{example}
  Let $r = 3$ and $\omega = e^{2\pi i/3}$.  Then
  $w = (
  1,
  \omega,
  \omega^2,
  2,
  2 \omega^2,
  \omega,
  1
  )$ is a $3$-word of shape $\lambda = ((2,1), (2), (1,1))$,
  with inverse images $w^*_{(0)} = (\{1,7\},\{4\})$.
  $w^*_{(1)} = (\{2,6\})$.
  $w^*_{(2)} = (\{3\},\{5\})$, omitting trailing empty sets.
  The canonical $r$-word of shape $\lambda$ is $w_{\lambda} = (
  1,
  1,
  2,
  \omega,
  \omega,
  \omega^2,
  2 \omega^2)
  $.
\end{example}

In general, we define the
\emph{canonical $r$-word} $w_{\lambda} \in A_r^n$ of shape $\lambda \vdash^r n$ as
\begin{align*}
  w_{\lambda} := w_{\lambda^{(0)}}^{(0)} w_{\lambda^{(1)}}^{(1)} \cdots w_{\lambda^{(r-1)}}^{(r-1)}\text,
\end{align*}
where, for $k = 0, \dots, r-1$, the word $w_{\lambda^{(k)}} = a_1 \cdots a_{n_k}$ is the canonical word of shape $\lambda^{(k)}$ and length $n_k = \Size{\lambda^{(k)}}$ over the alphabet $A$, as in Section~\ref{sec:A-words},
and $w_{\lambda^{(k)}}^{(k)} = (a_1, \omega^k) \cdots (a_{n_k}, \omega^k)$.
Denote by $X_{\lambda}$ the set of all rearrangements of $w_{\lambda}$.
Then $X_{\lambda}$ is isomorphic to the cosets of $\Grin_{\lambda}$ in $\Grin_n$, as $\Grin_n$-set.

We define an action of $\Delta^{(r)}_n$ on the $1$-dimensional space
$\C w_{\lambda}$ as follows.  Recall that $t_j \in \Grin_n$ is the diagonal matrix
$\diag(1, \dots, 1, \omega, 1, \dots 1)$ with entry $\omega$ in row $j$.
Then set $w_{\lambda}.t_j := \omega^k w_{\lambda}$ if the letter $w_j$ in $w_{\lambda}$ has phase $\omega^k$, i.e., if position $j$ belongs to $\lambda^{(k)}$.
By letting $\Symm_{\lambda}$ act trivially, this extends to an action of $\Grin_{\lambda}$ on $\C w_{\lambda}$.

Now let $V^{\lambda}$ be the induced $\Grin_n$-module
\begin{align*}
  V^{\lambda} := \Ind_{\Grin_{\lambda}}^{\Grin_n} (\C w_{\lambda})\text.
\end{align*}
Then, restricted to $\Symm_n$, $V^{\lambda}$ is the
permutation module with basis $X_{\lambda}$.
And for any $r$-word $x \in X_{\lambda}$, we have $x.t_j = \omega^k x$ if the letter $x_j$ of $x$ has phase $\omega^k$.  For $\sigma \in \Grin_n$, we denote by
$[\sigma]_{X_{\lambda}}^{X_{\lambda}}$ the matrix of the action of $\sigma$ on $V^{\lambda}$ relative to the basis $X_{\lambda}$. Then
\begin{align}\label{eq:mono-perm-mat}
  [v.\sigma]_{X_{\lambda}} = [v]_{X_{\lambda}} \, [\sigma]_{X_{\lambda}}^{X_{\lambda}}\text,
\end{align}
for all $v \in V^{\lambda}$, $\sigma \in \Grin_n$.

\subsection{Pairs of Multiwords and Multitableaus.}
\label{sec:multiword-pairs}

We consider the action of the symmetric group $\Symm_n$ on certain pairs of
words over the alphabet $A_r = A \times \mu_r$.
Set $A_r \boxtimes A_r := \coprod_k \{\pair{(a, \omega^k)}{(b, \omega^k)}: a, b \in A\}$, i.e., the pairs of letters in $A_r^2$ that have the same phase.
Set $A_r^n \boxtimes A_r^n := (A_r \boxtimes A_r)^n$, the pairs of words with the same phase on corresponding letters.
Then, for
$\lambda, \eta \vdash^r n$, let
\begin{align*}
  X_{\lambda} \boxtimes X_{\eta} := (X_{\lambda} \times X_{\eta}) \cap (A_r^n \boxtimes A_r^n)\text.
\end{align*}
Thus, a pair $\pair{x}{y} \in X_{\lambda} \boxtimes X_{\eta}$
can be regarded as a word over the alphabet
$A^2 \times \mu_r$.
Its \emph{inverse image function} is
\begin{align*}
  \pair{x}{y}^{*} \colon A^2 \times \mu_r \to 2^{\nnn}\text, \qquad
  (\pair{a}{b}, \omega^k) \mapsto (a, \omega^k).{x^{*}} \cap (b, \omega^k).{y^{*}}\text,
\end{align*}
which we identify with the $r$-tuple $(\pair{x}{y}^{*}_{(0)}, \dots, \pair{x}{y}^{*}_{(r-1)})$ of inverse images
\begin{align*}
  \pair{x}{y}^{*}_{(k)} \colon A^2  \to 2^{\nnn}, \qquad
  \pair{a}{b} \mapsto a.x^{*}_{(k)} \cap b.y^{*}_{(k)}
\end{align*}

As before, $\Symm_n$ acts on $X_{\lambda} \boxtimes X_{\eta}$ via
$ \pair{x}{y}.\sigma = \pair{x.\sigma}{y.\sigma}$,
and on their inverse images via right composition
\begin{align*}
  \pair{x}{y}^{*}.\sigma =   \pair{x}{y}^{*} \sigma =
  ((\pair{a}{b}, \omega^k) \mapsto (a.x^*_{(k)} \cap b.y^*_{(k)}).\sigma)
\end{align*}
in such a way that $(\pair{x}{y}.\sigma)^* = \pair{x}{y}^{*}.\sigma$.
We denote by
\begin{align*}
  (X_{\lambda} \boxtimes X_{\eta})^{\#}:= \{\pair{x}{y} \in X_{\lambda} \boxtimes X_{\eta} : \Symm(\pair{x}{y}^*) = 1\}
\end{align*}
the \emph{free component} of $X_{\lambda} \boxtimes X_{\eta}$.
\begin{proposition}\label{pro:mono-free-trans}
  Let $\lambda$ and $\eta$ be $r$-partitions of $n$.  Then $\Symm_n$ acts
  transitively on the free component of $X_{\lambda} \boxtimes X_{\eta}$
  if and only if $\eta = \lambda^t$.  In that case, the diagram of $\pair{x}{y} \in (X_{\lambda} \boxtimes X_{\eta})^{\#}$ is the Young diagram $D(\lambda)$ of $\lambda$.
\end{proposition}
\begin{proof}
  Component-wise.  Clearly, for $\Symm_n$ to act freely, $\lambda^{(k)}$ and
  $\eta^{(k)}$ need to be partitions of the same number $n^{(k)}$, for each
  $k$.  From Proposition~\ref{pro:symm-free-trans} it then follows that
  $\eta^{(k)} = (\lambda^{(k)})^t$, for each $k$.  Indeed, if
  $(\eta^{(k)})^t < \lambda^{(k)}$ for some $k$ then $\Symm_n$ does not act
  freely, and if $(\eta^{(k)})^t > \lambda^{(k)}$ for some $k$ then $\Symm_n$
  does not act transitively.
\end{proof}

Note that, if $\pair{x}{y} \in (X_{\lambda} \boxtimes X_{\lambda^t})^{\#}$ then
$\Size{\pair{a}{b}.\pair{x}{y}^*_{(k)}} = 1$ for all $(a, b) \in D(\lambda^{(k)})$, $k = 0, \dots, r-1$. So replacing each nonempty set $\pair{a}{b}.\pair{x}{y}^*_{(k)}$
by its single element, we can regard $T = \pair{x}{y}^{*}$ as a \emph{multitableau}
\begin{align*}
  T = (T^{(0)}, T^{(1)}, \dots, T^{(r-1)})
\end{align*}
of shape $\lambda$, where
\begin{align*}
  T^{(k)} \colon D(\lambda^{(k)}) \to \nnn\text, \quad
  (a, b) \mapsto i\text{, if } \pair{a}{b}.\pair{x}{y}^*_{(k)} = \{i\}\text,
\end{align*}
is a tableau of shape $\lambda^{(k)}$.  Denote by $T^{\flat}$ the
concatenation of the lists $(T^{(k)})^{\flat}$.  Then $T^{\flat}$ is a
permutation of $\nnn$.  We will write $\pair{x}{y}^{\flat}$ for $T^{\flat}$
if $T$ is the multitableau obtained from $\pair{x}{y}^*$, when
$\pair{x}{y} \in (X_{\lambda} \boxtimes X_{\lambda^t})^{\#}$.  Clearly, the
map $\pair{x}{y} \mapsto \pair{x}{y}^{\flat}$ is an equivariant bijection
between $(X_{\lambda} \boxtimes X_{\lambda^t})^{\#}$ and $\Symm_n$.

\begin{definition}
  Let $\lambda \vdash^r n$.  The \emph{Specht matrix} for $\lambda$ is the matrix $M_{\lambda} = (m^{\lambda}_{yx})$ with rows labelled by $y \in X_{\lambda^t}$
  and columns labelled by $x \in X_{\lambda}$, where
  \begin{align*}
    m^{\lambda}_{yx} =
    \begin{cases}
      \sgn{\pairx{x}{y}^{\flat}}& \text{if }
                       \pair{x}{y} \in (X_{\lambda} \boxtimes X_{\lambda^t})^{\#}, \\
      0, & \text{else.}
    \end{cases}
  \end{align*}
\end{definition}

\begin{remark}\label{rem:mono-myx}
  Note that $m^{\lambda}_{yx} = \sgn{\sigma}\, m^{\lambda}_{y.\sigma,x.\sigma}$,
  for all $\sigma \in \Symm_n$, $x \in X_{\lambda}$, and $y \in X_{\lambda^t}$.
\end{remark}

\subsection{Irreducible Modules.}
\label{sec:mono-modules}

Let $\lambda$ be an $r$-partition of $n$.  Recall that $V^{\lambda}$ is the
$\Grin_n$-module with basis $X_{\lambda}$, the rearrangements of $w_{\lambda}$.
For each $y \in X_{\lambda^t}$,
we define a vector $v_y \in V^{\lambda}$ as
\begin{align*}
  v_y = \sum_{x \in X_{\lambda}} m^{\lambda}_{yx}\, x\text.
\end{align*}

\begin{proposition}
  Let $S^{\lambda} = \Span{v_y : y \in X_{\lambda^t}}_{\C}$. Then
  $S^{\lambda}$ is an $\C \Grin_n$-module.  In fact, $S^{\lambda} = v_y \C \Grin_n$, for any $y \in X_{\lambda^t}$.
\end{proposition}

\begin{proof}
  Let $y \in X_{\lambda^t}$.  For $\sigma \in \Symm_n$, we have
  \begin{align*}
    v_y.\sigma = \sum_{x \in X_{\lambda}} m^{\lambda}_{yx}\, x.\sigma
    = \sgn{\sigma} \sum_{x \in X_{\lambda}} m^{\lambda}_{y.\sigma,x.\sigma}\, x.\sigma = v_{y.\sigma}\,
  \end{align*}
  by Remark~\ref{rem:mono-myx}.  Moreover, for $j \in \nnn$, we have
  \begin{align*}
    v_y.t_j  = \sum_{x \in X_{\lambda}} m^{\lambda}_{yx}\, x.t_j
    = \omega^k\, v_y\text,
  \end{align*}
  where $\omega^k$ is the phase of letter $y_j$ of $y$
  (and that of all the letters $x_j$ when $m^{\lambda}_{yx} \neq 0$).
\end{proof}

Using the notation from Section~\ref{sec:mono-irr},
for $k = 0,\dots, r-1$, let $n^{(k)} = \Size{\lambda^{(k)}}$, and let
$\underline{n} = (n^{(0)}, \dots, n^{(r-1)})$.
Let $\Symm_{\underline{n}} = \Grin_{\underline{n}} \cap \Symm_n = \Symm_{n^{(0)}} \times \dots \times \Symm_{n^{(r-1)}}$.
Let $X_{\lambda}^{\underline{n}}$ (resp.\ $X_{\lambda^t}^{\underline{n}}$) be the $\Symm_{\underline{n}}$-orbit
of $w_{\lambda}$ (resp. $w_{\lambda^t}$).
Then $X_{\lambda}^{\underline{n}}$ is the set of all $r$-words $x \in X_{\lambda}$
whose letters $x_i$ have phase $\omega^k$ for all $i \in I^{(k)}$, $k = 1, \dots, r-1$.  Likewise, $X_{\lambda^t}^{\underline{n}}$ is the set of all $r$-words $y \in X_{\lambda^t}$ whose letters $y_i$ have phase $\omega^k$ for all $i \in I^{(k)}$, $k = 1, \dots, r-1$.

Then, for each $y \in X_{\lambda^t}^{\underline{n}}$, the action of the normal subgroup
$\Delta_n^{(r)}$ on $\C v_y$ has character $\eta_{\underline{n}}$ from
\eqref{eq:mono-char-eta}.
Set
\begin{align*}
  S^{\lambda}_{\underline{n}} := \Span{v_y : y \in
  X_{\lambda^t}^{\underline{n}}}\text.
\end{align*}
Then, by construction (and omitting some of the details), $S^{\lambda}_{\underline{n}}$  is an
$\Grin_{\underline{n}}$-module
with character
$(\overline{\eta}_{n^{(0)}}^{(0)} \otimes \tilde{\chi}^{\lambda^{(0)}})
\boxtimes \dots \boxtimes (\overline{\eta}_{n^{(r-1)}}^{(r-1)} \otimes
\tilde{\chi}^{\lambda^{(r-1)}})$.  Finally, the induced $\Grin_n$-module
$S^{\lambda} = \Ind_{\Grin_{\underline{n}}}^{\Grin_n}
S^{\lambda}_{\underline{n}}$ has character $\chi^{\lambda}$, as defined in \eqref{eq:mono-chi-lambda}, and therefore is irreducible as $\Grin_n$-module.

\begin{theorem}
  The modules $S^{\lambda}$ for $\lambda \vdash^r n$ form a complete set of
  pairwise irreducible $\C\Grin_n$-modules.
\end{theorem}

\begin{proof}
  By~\eqref{eq:mono-irr}, the $\chi^{\lambda}$, $\lambda \in \Lambda^{(r)}_n$, are the irreducible characters of $\Grin_n$.
\end{proof}

\subsection{Representing Matrices and Standard Multitableaus.}
\label{sec:mono-mats}

Let $\lambda \vdash^r n$. An element $\sigma \in \Grin_n$ acts on
the $\Grin_n$-module $V^{\lambda}$ with basis $X_{\lambda}$ as a matrix
$[\sigma]_{X_{\lambda}}^{X_{\lambda}}$ according to~\eqref{eq:mono-perm-mat}.
We now identify a basis $B_{\lambda}$ of the submodule $S^{\lambda}$ and derive a formula
for the representing matrices $[\sigma]_{B_{\lambda}}^{B_{\lambda}}$ in terms of
the Specht matrix $M^{\lambda}$ and
$[\sigma]_{X_{\lambda}}^{X_{\lambda}}$ in Theorem~\ref{thm:mono-rep} below.

\begin{definition}
  A pair of words $\pair{x}{y} \in A_r^n \boxtimes A_r^n$ is called a \emph{standard pair} if, for $k = 0, \dots, r-1$, (i) for all $a \in A$, the restrictions
  of $y$ to the inverse images $a.x^*_{(k)}$, and
  (ii) for all $b \in A$, the restrictions of $x$ to the inverse images
  $b.y^*_{(k)}$, are strictly increasing.
\end{definition}

Note that, if $\pair{x}{y} \in X_{\lambda} \boxtimes X_{\lambda^t}$ is a
standard pair then $T = \pair{x}{y}^*$ is a standard multitableau, in the
sense that, in each component $T^{(k)}$ of $T$, the entries are increasing
along each row and each column, i.e., that $T^{(k)}$ is a standard tableau
(on the set of its entries).
We denote by $\SYT{\lambda}$ the set of all standard multitableaus of shape
$\lambda$.

Using the order $\omega^0 < \omega^1 < \dots < \omega^{r-1}$ on $\mu_r$,
we order $A_r$ first by phase then by radius, i.e., $(1, \omega^0) < (2, \omega^0) < \dots < (n, \omega^0) < (1, \omega^1) < (2, \omega^1) < \dotsm$.  We furthermore assume that both $A_r \boxtimes A_r$ and the set $(A_r \boxtimes A_r)^n$
of words of length $n$ over this set are ordered lexicographically.
This order on $A_r^n \boxtimes A_r^n$ induces an order on the multitableaus of shape $\lambda$, and in particular on $\SYT{\lambda}$.  With respect to this order, the
standard pairs $\pair{x}{y} \in (X_{\lambda} \boxtimes X_{\lambda^t})^{\#}$
are characterized by the property
\begin{align}\label{eq:smallest-multitab}
  \pair{x}{y} < \pair{x}{y}.\sigma\text{, for all } \sigma \in \Symm(x^*)
  \text{ and for all } \sigma \in \Symm(y^*).
\end{align}
It follows that the projection maps
$\pair{x}{y}^* \mapsto x$ and $\pair{x}{y}^* \mapsto y$  from $\SYT{\lambda}$ onto $X_{\lambda}$ and $X_{\lambda^t}$ are both injective.  We denote the
images of those projections by $X_{\lambda}^{\heartsuit}$ and $X_{\lambda^t}^{\heartsuit}$, respectively.

\begin{lemma}\label{la:mono-b-indep}
  The set $B_{\lambda} := \{v_y : y \in X_{\lambda^t}^{\heartsuit}\} \subseteq S^{\lambda}$ is linearly independent.
\end{lemma}

\begin{proof}
  Due to property~\eqref{eq:smallest-multitab}, the
  $X_{\lambda^t}^{\heartsuit} \times X_{\lambda}^{\heartsuit}$-submatrix
  $M_{\lambda}^{\heartsuit}$ of $M_{\lambda}$ is, up to signs,
  unitriangular and hence invertible.
\end{proof}

A counting argument shows that $B_{\lambda}$ is in fact a basis of $S^{\lambda}$:
Let $f_{\lambda} = \Size{B_{\lambda}} = \Size{\SYT{\lambda}}$.
Then $f_{\lambda} = \binom{n}{n^{(0)}, \dots, n^{(r-1)}}\, f_{\lambda^{(0)}} \dotsm  f_{\lambda^{(r-1)}}$.  For any $\underline{n} = (n^{(0)}, \dots, n^{(r-1)})$
with $n^{(0)} + \dots + n^{(r-1)} = n$, we thus have
\begin{align*}
  \sum_{\lambda^{(k)} \vdash n^{(k)}} f_{\lambda}^2 = \tbinom{n}{n^{(0)}, \dots, n^{(r-1)}}^2 \, \sum_{\lambda^{(0)} \vdash n^{(0)}} f_{\lambda_{(0)}}^2 \dotsm
  \sum_{\lambda^{(r-1)} \vdash n^{(r-1)}} f_{\lambda_{(r-1)}}^2
  = \tbinom{n}{n^{(0)}, \dots, n^{(r-1)}} \, n!\text,
\end{align*}
whence $\sum_{\lambda \vdash^r n} f_{\lambda}^2 = r^n \, n! = \Size{\Grin_n}$, by the Multinomial Theorem.   This leads to the main theorem of this section.

\begin{theorem}\label{thm:mono-rep}
  Let $\lambda \vdash^r n$.  Then
  \begin{align*}
    [\sigma]^{B_{\lambda}}_{B_{\lambda}} = (M_{\lambda}\, [\sigma]^{X_{\lambda}}_{X_{\lambda}})^{\heartsuit}\, (M_{\lambda}^{\heartsuit})^{-1}\text,
  \end{align*}
  for all $\sigma \in \Grin_n$.
\end{theorem}

\begin{proof}
  Analogous to the proof of Theorem~\ref{thm:symm-rep}.
\end{proof}

\begin{example}
  For $r=3$, $\lambda = (\emptyset, (3,2), \emptyset)$, and $\omega = e^{2\pi i/3}$, Theorem~\ref{thm:symm-rep} yields
  \begin{align*}
[t_1]_{B_{\lambda}}^{B_{\lambda}} &=
    {\tiny\arraycolsep3pt
    \left(
    \begin{array}{rrrrr}
   \omega & 0 & 0 & 0 & 0 \\
   0 & \omega & 0 & 0 & 0 \\
   0 & 0 & \omega & 0 & 0 \\
   0 & 0 & 0 & \omega & 0 \\
   0 & 0 & 0 & 0 & \omega \\
    \end{array}
    \right)}\text,&
[t_1\, (1,2,3,4,5)]_{B_{\lambda}}^{B_{\lambda}} &=
    {\tiny\arraycolsep3pt
    \left(
    \begin{array}{rrrrr}
   0 & \phantom{-}\omega & 0 & 0 & 0 \\
   0 & 0 & \omega & 0 & 0 \\
   \omega & 0 &-\omega &-\omega &-\omega \\
   0 & 0 &-\omega & 0 &-\omega \\
   0 & \omega & \omega & \omega & \omega \\
    \end{array}
    \right)}\text.
  \end{align*}
\end{example}


\section{Hyperoctahedral Groups}\label{sec:hypo}

We describe a construction of explicit matrices for the irreducible modules
of the hyperoctahedral group $\Hypo_n$.  As group, $\Hypo_n$ is isomorphic to
$\Grin_n$ for $r=2$, but the construction of the modules here is
substantially different from the one in Section~\ref{sec:mono}.  We will
reuse notation from the previous sections, sometimes with a new and different
meaning, as defined below.

\subsection{Bipartitions and Bidiagrams.}

A \emph{bipartition} is simply a $2$-partition $\lambda = (\lambda^{(0)}, \lambda^{(1)})$.  We say that $\lambda$ is a bipartition of $n$, and  write $\lambda \vdash^2 n$, if $\Size{\lambda^{(0)}} + \Size{\lambda^{(1)}} = n$.  And $\Lambda_n^{(2)}$ is the set of all bipartitions of $n$.

A \emph{bidiagram} is a $2 \times 2$-matrix of diagrams
$D = (\begin{smallmatrix}D_{00}&D_{01}\\D_{10}&D_{11}\end{smallmatrix})$.
The \emph{Young bidiagram} of the bipartition $\lambda$ is the bidiagram
\begin{align*}
  D(\lambda) =
  \left(\begin{array}{cc}\emptyset&D(\lambda^{(0)})\\D(\lambda^{(1)})&\emptyset\end{array}\right)\text.
\end{align*}
The \emph{transpose} of the bidiagram $D$ is the bidiagram
$D^t =
(\begin{smallmatrix}D_{00}^t&D_{10}^t\\D_{01}^t&D_{11}^t\end{smallmatrix})$.
The transpose of the bidiagram $D(\lambda)$ is the Young diagram of a
bipartition of $n$, the \emph{transpose}
$\lambda^t = ((\lambda^{(1)})^t, (\lambda^{(0)})^t)$ of the bipartition $\lambda$.

\begin{example}
  The bipartition $\lambda = ((2,1), (2))$ and its transpose $\lambda^t = ((1,1), (2,1))$ have bidiagrams
  \begin{align*}
  D(\lambda) &=
  \left(\begin{array}{cc}\emptyset&\smash{\raisebox{4pt}{\tiny\ydiagram{2,1}}}\\\smash{\tiny\ydiagram{2}}&\emptyset\end{array}\right)\text,&
  D(\lambda)^t = D(\lambda^t) &=
  \left(\begin{array}{cc}\emptyset&\smash{\raisebox{4pt}{\tiny\ydiagram{1,1}}}\\\smash{\raisebox{4pt}{\tiny\ydiagram{2,1}}}&\emptyset\end{array}\right)\text.
  \end{align*}
\end{example}

\subsection{Signed Permutations and Cycle Type.}
\label{sec:hypo-classes}

A \emph{signed permutation} of $n$ points is a permutation $\sigma$ of the
set $\pm\nnn$ with the property that $(-i).{\sigma} = -(i.{\sigma})$ for all
$i \in \nnn$.  A signed permutation $\sigma$ is thus completely determined by
the images $i.{\sigma}$, $i \in \nnn$, and we can regard $\sigma$
as a function $\sigma \colon \nnn \to \pm \nnn$.

Such a function $\sigma$ corresponds to a monomial matrix in $\Grin_n$ for
$r = 2$.
We write $\sgn{\sigma}$ for the \emph{sign} of the signed permutation $\sigma$, i.e., the determinant of the corresponding monomial matrix.
The \emph{cycle type} of $\sigma$ is the $2$-partition $\lambda$
which is the cycle type of the corresponding matrix in $\Symm_n^{(2)}$.
Then two signed permutations in $\Hypo_n$ are conjugate in $\Hypo_n$
if and only if they have the same bipartition as their cycle type.
Hence the set $\Lambda_n^{(2)}$ parameterizes both the conjugacy classes and the irreducible representations of $\Hypo_n$.

\subsection{Biwords and Bipartitions.}\label{sec:biwords}

As before, let $A = \nnn$.
Let $\mathbold{A} = A \times \{-1, 0, +1\} = A_{-} \sqcup A_{\circ} \sqcup A_{+}$,
where $A_{-} = A \times \{-1\}$, $A_{\circ} = A \times \{0\}$ and $A_{+} = A \times \{+1\}$. We set $A_{\pm} = A_{-} \sqcup A_{+}$.  For $i \in \Z$, we write $\overline{i}$ for $-i$.
For $\mathbold{a} = (a,{\epsilon}) \in \mathbold{A}$,
we call $a$ the \emph{radius} and $\epsilon$ the \emph{phase} of $\mathbold{a}$.
We also define a \emph{bar involution} as
\begin{align*}
  \overline{(a, {\epsilon})} := (a, \overline{\epsilon})\text.
\end{align*}
An element $w \in \mathbold{A}^n$ is called a \emph{biword} of length $n$.
We regard the biword $w$ as a map
\begin{align*}
w \colon \pm\nnn \to \mathbold{A}
\end{align*}
by setting
$\overline{i}.w := \overline{i.w}$ for $i \in \nnn$.  Then
its \emph{inverse image map}
\begin{align*}
  w^{*} \colon \mathbold{A} \to 2^{\pm\nnn}\text,\quad
  \mathbold{a} \mapsto \{i \in \pm\nnn : i.w = \mathbold{a})\}\text,
\end{align*}
can be identified with the pair $(w^{*}_{(0)}, w^{*}_{(1)})$ of maps
\begin{align*}
  w^{*}_{(0)} \colon A \to 2^{\pm\nnn}\text, \quad  a \mapsto (a, 0).{w^{*}}\text,\\
  w^{*}_{(1)} \colon A \to 2^{\pm\nnn}\text, \quad  a \mapsto (a, 1).{w^{*}}\text.
\end{align*}
Note that $(a, -1).{w^{*}} = \overline{a.{w^{*}_{(1)}}}$,
and that $a.{w^{*}_{(0)}} = \overline{a.{w^{*}_{(0)}}}$, for all $a \in A$.

The group $\Hypo_n$ acts on the set $\mathbold{A}^n$ of biwords of length $n$
by \emph{inverse left composition}. If $w \colon \pm\nnn \to \mathbold{A}$, $i \mapsto w_i$, then
\begin{align*}
  w.\sigma = \sigma^{-1}w = (i \mapsto w_{i.\sigma^{-1}})\text,
\end{align*}
for $\sigma \in \Hypo_n$.
In particular, the sign change $t_i$
 acts by replacing the $i$th letter
 $w_i$ of $w$ by $\overline{w_i}$.  Moreover, using the action
$J.\sigma = \{j.\sigma : j \in J\}$ of $\Hypo_n$ on the subsets $J \subseteq \pm\nnn$,
the group $\Hypo_n$
acts by \emph{right composition} on the inverse images:
\begin{align*}
  w^{*}.\sigma = w^{*} \sigma = (a \mapsto a.{w^{*}\sigma} = (a.{w^{*}}).\sigma)\text,
\end{align*}
in such a way that the bijection $w \mapsto w^{*}$ is $\Hypo_n$-equivariant:
\begin{align*}
  w^{*}.\sigma
  = w^{*} \sigma = (\sigma^{-1}w)^{*}
  = (w.\sigma)^{*}\text.
\end{align*}
We denote the \emph{stabilizer} in $\Hypo_n$ of the biword $w \in \mathbold{A}^n$ by~$\Hypo(w^{*})$ and say that the $\Hypo_n$-orbit of $w$ consists of all its \emph{signed rearrangements}.

Let  $\lambda = (\lambda^{(0)}, \lambda^{(1)})$ be a bipartition of $n$.  We define
the \emph{canonical biword} of shape $\lambda$ as
\begin{align*}
  w_{\lambda} = w_{\lambda^{(0)}}^{\circ} w_{\lambda^{(1)}}^{+}\text,
\end{align*}
where $w_{\lambda^{(0)}} = a_1 \cdots a_{n_0}$ is the canonical word of shape $\lambda^{(0)}$ and length $n_0 = \Size{\lambda^{(0)}}$ over the
alphabet $A$, as defined in Section~\ref{sec:A-words}, and
$w_{\lambda^{(0)}}^{\circ} = (a_1, 0) \cdots (a_{n_0}, 0)$,
and where $w_{\lambda^{(1)}} = b_1 \cdots b_{n_1}$ is the canonical word of shape $\lambda^{(1)}$ and length $n_1 = \Size{\lambda^{(1)}}$ over the alphabet $A$, and $w_{\lambda^{(1)}}^{+} = (b_1, 1) \cdots (b_{n_1}, 1)$.
Then the stabilizer of
$w_{\lambda}$ in $\Hypo_n$ is $\Hypo_{\lambda} = \Hypo_{\lambda^{(0)}} \times \Symm_{\lambda^{(1)}}$,
where, for a partition $\eta = (\eta_1, \dots, \eta_l)$, $\Hypo_{\eta} =
\Hypo_{\eta_1} \times \dots \times
\Hypo_{\eta_l}$.
Denote by $X_{\lambda}$
the $\Hypo_n$-orbit of $w_{\lambda}$.
Then $X_{\lambda}$ is isomorphic to the (right) cosets of $\Hypo_{\lambda}$
in $\Hypo_n$, as $\Hypo_n$-set.  For $\sigma \in \Hypo_n$, we denote by $[\sigma]_{X_{\lambda}}^{X_{\lambda}}$ the matrix of the action of $\sigma$ on the
permutation module $V^{\lambda}$ with basis $X_{\lambda}$.  Then
\begin{align}\label{eq:hypo-perm-mat}
  [v.\sigma]_{X_{\lambda}} = [v]_{X_{\lambda}}\, [\sigma]_{X_{\lambda}}^{X_{\lambda}},
\end{align}
for all $v \in V^{\lambda}$, $\sigma \in \Hypo_n$.

\subsection{Pairs of Biwords and Bitableaus.}

We consider
the action of $\Hypo_n$ on certain
pairs of words $\pair{x}{y}$ over the alphabet $\mathbold{A} = A_{\circ} \cup A_{\pm}$.  Set $\mathbold{A} \mathbin{\square} \mathbold{A} := A_{\circ} \times A_{\pm} \cup A_{\pm} \times A_{\circ}$, i.e., the pairs of letters $\pair{\mathbold{a}}{\mathbold{b}} \in \mathbold{A}^2$ which have unlike phases, in the sense that one is $0$ and the other is not.
Set  $\mathbold{A}^n \mathbin{\square} \mathbold{A}^n := (\mathbold{A} \mathbin{\square} \mathbold{A})^n$, the pairs of words with unlike phase on corresponding letters.
Then, for $\lambda, \eta \vdash^2 n$, let
\begin{align*}
  X_{\lambda} \mathbin{\square} X_{\eta} := (X_{\lambda} \times X_{\eta}) \cap (\mathbold{A}^n \mathbin{\square} \mathbold{A}^n)\text.
\end{align*}
The inverse image $\pair{x}{y}^* \colon \mathbold{A} \mathbin{\square} \mathbold{A} \to 2^{\pm\nnn}$ of a pair $\pair{x}{y} \in X_{\lambda} \mathbin{\square} X_{\eta}$ can then be identified with the pair of maps
\begin{align*}
  \pair{x}{y}^{*}_{(01)} &\colon A^2 \to 2^{\pm\nnn}, \quad
  \pair{a}{b} \mapsto a.{x^{*}_{(0)}} \cap b.{y^{*}_{(1)}}\text,\\
  \pair{x}{y}^{*}_{(10)} &\colon A^2 \to 2^{\pm\nnn}, \quad
  \pair{a}{b} \mapsto a.{x^{*}_{(1)}} \cap b.{y^{*}_{(0)}}\text,
\end{align*}
noting that, since $\overline{\mathbold{a}}.x^{*} = \overline{\mathbold{a}.x^{*}}$, for
$\mathbold{a} \in \mathbold{A}$, the pair $\pair{x}{y}$ can be recovered from
the two maps   $\pair{x}{y}^{*}_{(01)}$ and   $\pair{x}{y}^{*}_{(10)}$.

Now, $\Hypo_n$ acts on $X_{\lambda} \mathbin{\square} X_{\eta}$
via $\pair{x}{y}.\sigma = \pair{x.\sigma}{y.\sigma}$, and on their inverse images $\pair{x}{y}^*$ via right composition,
\begin{align*}
  \pair{x}{y}_{\epsilon\delta}^*.\sigma =
  \pair{x}{y}_{\epsilon\delta}^* \sigma =
  (\pair{a}{b} \mapsto a.x^*_{(\epsilon)} \cap b.y^*_{(\delta)}).\sigma\text,
\end{align*}
in such a way that $(\pair{x}{y}.\sigma)^* = \pair{x}{y}^*.\sigma$.
We denote by $\Hypo(\pair{x}{y}^*)$ the stabilizer of $\pair{x}{y}^*$ (and of $\pair{x}{y}$) in $\Hypo_n$, and by
\begin{align*}
  (X_{\lambda} \mathbin{\square} X_{\eta})^{\#}:= \{\pair{x}{y} \in X_{\lambda} \mathbin{\square} X_{\eta} : \Hypo(\pair{x}{y}^*) = 1\}
\end{align*}
the \emph{free component} of $X_{\lambda} \mathbin{\square} X_{\eta}$.

\begin{proposition}\label{pro:hypo-free-trans}
  Let $\lambda$ and $\eta$ be bipartitions of $n$.  Then $\Hypo_n$ acts
  transitively on the free component of $X_{\lambda} \mathbin{\square} X_{\eta}$
  if and only if $\eta = \lambda^t$.  In that case, the diagram of $\pair{x}{y} \in (X_{\lambda} \mathbin{\square} X_{\eta})^{\#}$ is the Young diagram $D(\lambda)$ of $\lambda$.
\end{proposition}
\begin{proof}
  Component-wise.  Clearly, for $X_{\lambda} \mathbin{\square} X_{\eta}$
 not to be empty,
  $\lambda^{(0)}$ and
  $\eta^{(1)}$ need to be partitions of the same number $n^{(0)}$, while
  $\lambda^{(1)}$ and $\eta^{(0)}$ are partitions of $n^{(1)} = n - n^{(0)}$.
Moreover,
 for $\Hypo_n$ to act freely, we need
  $\eta^{(0)} = (\lambda^{(1)})^t$  and
  $\eta^{(1)} = (\lambda^{(0)})^t$, by Proposition~\ref{pro:symm-free-trans}.
\end{proof}

Note that, if
$\pair{x}{y} \in (X_{\lambda} \mathbin{\square} X_{\lambda^t})^{\#}$ then
$\Size{\pair{a}{b}.\pair{x}{y}^*_{\epsilon\delta}} = 1$ for all
$(a, b) \in D(\lambda)_{\epsilon\delta}$.  So replacing each nonempty
preimage $\pair{a}{b}.\pair{x}{y}^*$ by its single element, we can regard
$T = \pair{x}{y}^*$ as a \emph{bitableau}
\begin{align*}
  T =
  \left(\begin{array}{cc}\emptyset&T^{(01)}\\T^{(10)}&\emptyset\end{array}\right)
\end{align*}
of shape $\lambda$, where
\begin{align*}
  T^{(01)} \colon D(\lambda^{(0)}) \to \pm[n], \quad
  (a, b) \mapsto i\text{, if } \pair{a}{b}.\pair{x}{y}^*_{(01)} =\{i\}\text,\\
  T^{(10)} \colon D(\lambda^{(1)}) \to \pm[n], \quad
  (a, b) \mapsto i\text{, if } \pair{a}{b}.\pair{x}{y}^*_{(10)} =\{i\}\text,
\end{align*}
are tableaus of shape $\lambda^{(0)}$ and $\lambda^{(1)}$, respectively.
Denote by $T^{\flat}$ the
concatenation of the lists $(T^{(01)})^{\flat}$ and  $(T^{(01)})^{\flat}$.  Then $T^{\flat}$ is a
signed permutation of $\nnn$.  We will write $\pair{x}{y}^{\flat}$ for $T^{\flat}$
if $T$ is the bitableau obtained from $\pair{x}{y}^*$, when
$\pair{x}{y} \in (X_{\lambda} \mathbin{\square} X_{\lambda^t})^{\#}$.  Clearly, the
map $\pair{x}{y} \mapsto \pair{x}{y}^{\flat}$ is an equivariant bijection
between $(X_{\lambda} \mathbin{\square} X_{\lambda^t})^{\#}$ and $\Hypo_n$.

\begin{definition}
  Let $\lambda \vdash^2 n$.  The \emph{Specht matrix} for $\lambda$ is the matrix $M_{\lambda} = (m^{\lambda}_{yx})$ with rows labelled by $y \in X_{\lambda^t}$
  and columns labelled by $x \in X_{\lambda}$, where
  \begin{align*}
    m^{\lambda}_{yx} =
    \begin{cases}
      \sgn{\pairx{x}{y}^{\flat}}\text, & \text{if }
                       \pair{x}{y} \in (X_{\lambda} \mathbin{\square} X_{\lambda^t})^{\#}\text, \\
      0\text, & \text{else.}
    \end{cases}
  \end{align*}
\end{definition}

\begin{remark}\label{rem:hypo-myx}
  Note that
    $m_{yx}^{\lambda} = \sgn{\sigma}\, m_{y.\sigma,x.\sigma}^{\lambda}$
  for all $\sigma \in \Hypo_n$, $x \in X_{\lambda}$, and $y \in X_{\lambda^t}$.
\end{remark}

\subsection{Irreducible Modules.}
\label{sec:hypo-irr-mod}

Let $\lambda$ be a bipartition of $n$.  Recall that $V^{\lambda}$ is the
permutation module with basis $X_{\lambda}$, the signed rearrangements of $w_{\lambda}$.  For each word $y \in X_{\lambda^t}$, we set
\begin{align*}
  c_y = \sum_{\sigma \in \Hypo_n} \sgn{\sigma}\, \sigma \in \C \Hypo_n\text.
\end{align*}
Then, for each bitableau $T = \pair{x}{y}^{*}$ of shape $\lambda$,
we define a vector $v_T \in V^{\lambda}$ as $v_T := \varepsilon_{T^{\flat}}\, x.c_y$.  Then
\begin{align*}
  v_T = \sum_{\sigma \in \Hypo(y^*)} \sgn{(T.\sigma)^{\flat}}\, x.\sigma
  = \sum_{x'} \sgn{\pair{\smash{x'}}{y}^{\flat}}\, x'
  = \sum_{x' \in X_{\lambda}} m^{\lambda}_{yx'}\, x'\text,
\end{align*}
where the second sum is over all signed rearrangements $x'$ of $x$ such that
$\pair{x'}{y} \in (X_{\lambda} \mathbin{\square} X_{\lambda^t})^{\#}$.
It follows that $v_T$ only depends on $y \in X_{\lambda^t}$, and that the
coefficient vector $[v_T]_{X_{\lambda}} = (m^{\lambda}_{yx})_{x \in X_{\lambda}}$ is the $y$-row of the Specht matrix $M^{\lambda}$.  We set $v_y := v_T$ for any bitableau $T = \pair{x}{y}^*$.

\begin{proposition}\label{pro:hypo-S-lambda}
  Let $S^{\lambda} := \Span{v_y : y \in X_{\lambda^t}}_{\C}$.  Then
  $S^{\lambda}$ is a $\C \Hypo_n$-module.  In fact,   $S^{\lambda} = v_y \C \Hypo_n$, for any  $y \in X_{\lambda^t}$.
\end{proposition}

\begin{proof} For each $\sigma \in \Hypo_n$ and each $y \in X_{\lambda^t}$, we have
  \begin{align*}
    v_y.\sigma
    = \sum_{x\in X_{\lambda}} m_{yx}^{\lambda}\, x.\sigma
    = \sgn{\sigma} \sum_{x\in X_{\lambda}} m_{y.\sigma,x.\sigma}^{\lambda}\, x.\sigma
    = \sgn{\sigma} v_{y.\sigma}\text,
  \end{align*}
by Remark~\ref{rem:hypo-myx}.
\end{proof}

We order the set $\Lambda^{(2)}_n$ of bipartitions $\lambda = (\lambda^{(0)}, \lambda^{(1)})$ of $n$ first by $\Size{\lambda^{(0)}}$, then by $\lambda^{(0)}$
using the lexicographic order from Section~\ref{sec:A-part}, and then by $\lambda^{(1)}$, using the same lexicographic order of partitions.

\begin{example}
  On $\Lambda^{(2)}_3$, writing $\lambda^{(0)}.\lambda^{(1)}$ for $\lambda = (\lambda^{(0)}, \lambda^{(1)})$ and omitting commas and parentheses in $\lambda^{(0)}$ and $\lambda^{(1)}$, we have:
  \begin{align*}
    .111 < .21 < .3 < 1.11 < 1.2 < 11.1 < 2.1 < 111. < 21. < 3.
  \end{align*}
\end{example}

Suppose now that $\lambda, \lambda' \in \Lambda^{(2)}_n$ are such that
$\Size{\lambda^{(0)}} < \Size{\lambda'^{(0)}}$. Let $x \in X_{\lambda'}$ and
$y \in X_{\lambda^t}$.  Then $\Size{\lambda'^{(0)}}$ is the number of letters
of the biword $x$ contained in $A_{\circ}$, and $\Size{\lambda^{(0)}}$ is the
number of letters of the biword $y$ contained in $A_{\pm}$.  It follows that
there is an index $i \in \nnn$ such that both $x_i, y_i \in A_{\circ}$.
Hence the stabilizer $\Hypo(y^*)$ contains the sign change $t_i$ with
determinant $\sgn{t_i} = -1$, which also fixes $x$.  As in the proof
of Corollary~\ref{cor:symm-ny-action}, it follows that $x.c_y = 0$.  Overall, arguing as in that proof, we get an analogous result here.

\begin{corollary}\label{cor:hypo-ny-action}
  Let $\lambda, \lambda' \vdash^2 n$ be such that $\lambda < \lambda'$.  Then, for all $y \in X_{\lambda^t}$,
  \begin{enumerate}
  \item [(i)] $V^{\lambda}.c_y = S^{\lambda}.c_y = \C v_y \neq 0$, and
  \item [(ii)] $V^{\lambda'}.c_y = S^{\lambda'}.c_y = 0$.
  \end{enumerate}
\end{corollary}

\begin{theorem}\label{thm:hypo-complete}
  The modules $S^{\lambda}$ for $\lambda \vdash^2 n$ form a complete set of
  pairwise non-isomorphic irreducible $\C \Hypo_n$-modules.
\end{theorem}

\begin{proof}
  Analogous to the proof of Theorem~\ref{thm:symm-complete}.
\end{proof}

\subsection{Representing Matrices and Standard Bitableaus.}
\label{sec:hypo-mats}

Let $\lambda \vdash^2 n$.
An element $\sigma \in \Hypo_n$ acts on
the $\Hypo_n$-module
$V^{\lambda}$ with basis $X_{\lambda}$ as a permutation matrix
$[\sigma]_{X_{\lambda}}^{X_{\lambda}}$ according to \eqref{eq:hypo-perm-mat}.
We now identify a basis $B_{\lambda} \subseteq V^{\lambda}$ of the submodule $S^{\lambda}$ and derive a formula for the representing matrices $[\sigma]_{B_{\lambda}}^{B_{\lambda}}$ in terms of the Specht matrix $M_{\lambda}$ and $[\sigma]_{X_{\lambda}}^{X_{\lambda}}$ as Theorem~\ref{thm:hypo-rep} below.

\begin{definition}
  A pair of biwords $\pair{x}{y} \in \mathbold{A}^n \mathbin{\square} \mathbold{A}^n$ is called a \emph{standard pair} if $x, y \in (A_{\circ} \sqcup A_{+})^n$ and
  (i) for all $a \in A$, the restrictions of $y$ to $a.x_{(0)}^* \cap [n]$ and to $a.x_{(1)}^*$
  are strictly increasing in radius, and
  (ii) for all $b \in A$, the restrictions of $x$ to $b.y_{(0)}^* \cap [n]$ and to $b.y_{(1)}^*$
  are strictly increasing in radius.
\end{definition}

Note that, if $\pair{x}{y} \in X_{\lambda} \mathbin{\square} X_{\lambda^t}$ is a standard pair then $T = \pair{x}{y}^*$ is a \emph{standard bitableau}, in the sense that the entries in each of the two components are positive, and increasing along each row and column.
We denote by $\SYT{\lambda}$ the set of all standard bitableaus of shape $\lambda$.

Using the order $(1,0) < (1,1) < (1,-1) < (2,0) < (2,1) < (2,-1) < \dotsm$ on $\mathbold{A}$, we assume that both $\mathbold{A} \mathbin{\square} \mathbold{A}$ and the set $(\mathbold{A} \mathbin{\square} \mathbold{A})^n$ of words of length $n$ over this set are ordered lexicographically.  This order induces an order on the
bitableaus of shape $\lambda$, and in particular on $\SYT{\lambda}$.  With respect to this order, the standard pairs $\pair{x}{y} \in (X_{\lambda} \mathbin{\square} X_{\lambda^t})^{\#}$ are characterized by the property
\begin{align}\label{eq:smallest-bitab}
  \pair{x}{y} < \pair{x}{y}.\sigma\text{, for all } \sigma \in \Hypo(x^*)
  \text{ and for all } \sigma \in \Hypo(y^*).
\end{align}
It follows that the projection maps $\pair{x}{y}^* \mapsto x$ and $\pair{x}{y}^* \mapsto y$ from $\SYT{\lambda}$ onto $X_{\lambda}$ and $X_{\lambda^t}$ are both injective.
We denote the images of those projections by $X_{\lambda}^{\heartsuit}$ and $X_{\lambda^t}^{\heartsuit}$, respectively, and set
\begin{align*}
  B_{\lambda} = \{v_y : y \in X_{\lambda^t}^{\heartsuit}\}
  = \{ v_T : T \in \SYT{\lambda} \}\text.
\end{align*}

\begin{lemma}\label{la:hypo-b-indep}
  The set $B_{\lambda} \subseteq S^{\lambda}$ is linearly independent.
\end{lemma}

\begin{proof}
  Due to property~\eqref{eq:smallest-bitab}, the
  $X_{\lambda^t}^{\heartsuit} \times X_{\lambda}^{\heartsuit}$-submatrix
  $M_{\lambda}^{\heartsuit}$ of $M_{\lambda}$ is, up to signs,
  unitriangular and hence invertible.
\end{proof}

As in Section~\ref{sec:mono-mats}, from $\Size{B_{\lambda}} = \Size{\SYT{\lambda}}$, it follows that $B_{\lambda}$ is in fact a basis of $S^{\lambda}$.  This
leads to the main theorem of this section.

\begin{theorem}\label{thm:hypo-rep}
  Let $\lambda \vdash^2 n$.  Then
  \begin{align*}
    [\sigma]^{B_{\lambda}}_{B_{\lambda}} = (M_{\lambda}\, [\sigma]^{X_{\lambda}}_{X_{\lambda}})^{\heartsuit}\, (M_{\lambda}^{\heartsuit})^{-1}\text,
  \end{align*}
  for all $\sigma \in \Hypo_n$.
\end{theorem}

\begin{proof}
  Analogous to the proof of Theorem~\ref{thm:symm-rep}.
\end{proof}

\begin{example}
  For $\lambda = (\emptyset, (3,2))$, Theorem~\ref{thm:symm-rep} yields
  \begin{align*}
[t_1]_{B_{\lambda}}^{B_{\lambda}} &=
    {\tiny\arraycolsep3pt
    \left(
    \begin{array}{rrrrr}
   -1 & 0 & 0 & 0 & 0 \\
   0 & -1 & 0 & 0 & 0 \\
   0 & 0 & -1 & 0 & 0 \\
   0 & 0 & 0 & -1 & 0 \\
   0 & 0 & 0 & 0 & -1 \\
    \end{array}
    \right)}\text,&
[t_1\, (1,2,3,4,5)]_{B_{\lambda}}^{B_{\lambda}} &=
    {\tiny\arraycolsep3pt
    \left(
    \begin{array}{rrrrr}
   0 & -1 & 0 & 0 & 0 \\
   0 & 0 & -1 & 0 & 0 \\
   -1 & 0 & 1 & 1 & 1 \\
   0 & 0 & 1 & 0 & 1 \\
   0 & -1 & -1 & -1 & -1 \\
    \end{array}
    \right)}\text.
  \end{align*}
\end{example}


\bibliography{buntspecht}
\bibliographystyle{amsplain}

\end{document}